\documentclass[12pt,a4paper]{article}
\usepackage[utf8]{inputenc}
\usepackage{amsmath}
\usepackage{amsfonts}
\usepackage{amssymb}
\usepackage{amsthm}
\usepackage[left=2.5cm,right=2.5cm,top=2cm,bottom=2.5cm]{geometry}

\usepackage[svgnames,hyperref]{xcolor}
\newcommand{\mycolor}{Navy}
\usepackage[pdfauthor={S. Dinew, Son H. Do and  Dat T. T\^O}, colorlinks, linktocpage, citecolor = \mycolor, linkcolor = \mycolor, urlcolor = \mycolor]{hyperref}

\newtheorem{theorem}{Theorem}
\newtheorem{lemma}{Lemma}
\newtheorem{definition}{Definition}
\newtheorem{proposition}{Proposition}
\newtheorem{corollary}{Corollary}

\usepackage{enumitem}

\newcommand{\al}{\alpha}

\newcommand{\e}{\varepsilon}
\newcommand{\om}{\omega}
\newcommand{\f}{\varphi}

\newcommand{\s}{\sigma}

\newcommand{\G}{\Gamma}
\newcommand{\Om}{\Omega}
\newcommand{\gk}{g_{\xi}}
\newcommand{\pa}{\partial}
\newcommand{\la}{\lambda}
\newcommand{\lbr}{\lbrace}
\newcommand{\rbr}{\rbrace}

\newcommand{\C}{{\mathbb{C}}}
\newcommand{\N}{{\mathbb{N}}}

\newcommand{\R}{{\mathbb{R}}}

\begin{document}
\begin{center}
{\Large \bf  A viscosity approach to the Dirichlet problem for degenerate complex Hessian type equations}
\bigskip\bigskip

\centerline{S\l awomir Dinew,  Hoang-Son Do and Tat Dat T\^o\footnote{The first author was supported by the NCN grant 2013/08/A/ST1/00312. The second author was supported in part by the Vietnam National Foundation for Science and Technology Development (NAFOSTED) under grant number 101.02-2017.306. The third author was supported by the CFM foundation. }}
\end{center}

\begin{abstract}
A viscosity approach is introduced for the Dirichlet problem associated to complex Hessian type equations  on domains in $\C^n$. The arguments are  modelled on the theory of viscosity solutions for real Hessian type equations developed by Trudinger \cite{Tr90}.  As consequence we solve the Dirichlet problem for the Hessian quotient and special Lagrangian equations. We also establish basic regularity results for the solutions. 
\end{abstract}
\tableofcontents

\section{Introduction}
Partial differential equations play pivotal role in modern complex geometric analysis. Their applications typically involve a geometric problem which can be reduced to the solvability of an associated equation. This solvability can be deducted by various methods yet most of the basic approaches exploit a priori estimates for suitably defined weak solutions. Thus although geometers work in the smooth category, the associated weak theory plays an important role.

One of the most successful such theories is the  pluripotential theory associated to the complex Monge-Amp\`ere eqution developed by Bedford and Taylor \cite{BT1,BT2}, Ko\l odziej \cite{K98}, Guedj and Zeriahi \cite{GZ05} and many others. Roughly speaking pluripotential theory allows to define $(i\partial\bar{\partial}u)^k$ as a measure valued positive closed differential form (i.e. a closed positive current) for any locally bounded plurisubharmonic function which in turn allows to deal with non smooth weak solutions of Monge-Amp\`ere equations. Unfortunately the pluripotential approach is  applicable only for a limited class of nonlinear operators, such as the $m$-Hessian equations- see \cite{DK, Lu}. 

Some of the most important examples on nonlinear operators for which pluripotential tools do not seem to apply directly are the complex Hessian quotient operators. These are not only interesting for themselves but also appear in interesting geometrical problems. One such example is the  Donaldson equation that we describe below.

Given a compact K\"ahler manifold $(X,\om)$ equipped with another K\"ahler form $\chi$ one seeks another K\"ahler form $\tilde{\chi}$ cohomologous to $\chi$ such that
\begin{equation}\label{Donaldson}
\om\wedge\tilde{\chi}^{n-1}=c\tilde{\chi}^n
\end{equation}
with the constant $c$ dependent only on the cohomology classes of $\chi$ and $\om$.

In \cite{D} Donaldson  introduced this  equation in order to study the properness of the Mabuchi functional. Its parabolic version known as the $J$-flow was introduced indepenttly  by Donaldson \cite{D}  and Chen \cite{Ch} and investigated afterwards  by Song and Weinkove \cite{W1,W2},\cite{SW}. It is known that the equation (\ref{Donaldson}) is not always solvable (\cite{SW}) and a conjecture of Lajmi and Sz\'ekelyhidi \cite{LSz} predicts that its solvability  is linked to positivity of certain integrals.  It was proved that, in general, these positivity conditions are equivalent to the existence  of $C$-subsolution of Sz\'ekelyhidi \cite{Sze1}, and also the existence of parabolic $C$-subsolution  for the corresponding flows (cf. \cite{PT}).  It would be helpful to study the boundary case when we only have  nonnegativity conditions (see \cite{FLSW} for Donaldson equation on surfaces). For its resolution it seems crucial to develop the associated theory of weak solutions for the given Hessian quotient equation. A major problem in applying some version of pluripotential theory for this equation is that essentially one has to define the quotient of two measure valued operators. 

In order to circumvent this difficulty one can look for possibly different theory of weak solutions. One such approach, known as the viscosity method was invented long ago in the real setting \cite{CIL}, but was only recently introduced for complex Monge-Amp\`ere equations by Eyssidieux-Guedj-Zeriahi \cite{EGZ}, Wang \cite{Wa12} and Harvey-Lawson \cite{HL}.

In the current note we initiate the viscosity theory for general complex nonlinear elliptic PDEs. As the manifold case is much harder we focus only on the local theory i.e. we deal with functions defined over domains in $\C^n$. Nevertheless we wish to point out that nonlinear PDEs appear also in geometric problems which are defined over domains in $\C^n$- see for example \cite{CPW}, where a Dirichlet problem for the special Lagrangian type equation is studied. We also illustrate in Section \ref{lagrangian_section} that  our method can be applied to solve the  Dirichlet problem for the special degenerate Lagrangian type equation.

In our investigations we heavily rely on the corresponding real theory developed by Trudinger in \cite{Tr90}. Some of our results can be seen as complex analogues of the real results that can be found there. In particular we have focused on various comparison principles in Section 3, which we  use later on  to study existence, uniqueness and regularity of the associated Dirichlet problems. One of our main result is the sharp regularity for viscosity solutions to the Dirichlet problem for a very general class of operators including Hessian quotient type equations.

Another interesting topic is the comparison between viscosity and pluripotential theory whenever the latter can be reasonably defined. A guiding principle for us is the basic observation made by Eyssidieux, Guedj and Zeriahi \cite{EGZ} that plurisubharmonic functions correspond to viscosity subsolutions to the complex Monge-Amp\`ere equation.  We prove several analogous results for general complex nonlinear operators. It has to be stressed that the notion of a supersolution, which does not appear in pluripotential theory, is a very subtle one for nonlinear elliptic PDEs and several alternative definitions are possible. We in particular compare these and introduce a notion of supersolution that unifies the previously known approaches.

A large part of the note is devoted to complex Hessian quotient equations in domains in $\C^n$. One of our goals in this case was to initiate the construction of the undeveloped  pluripotential theory associated to such equations. We rely on connections with the corresponding viscosity theory. Our findings yield in particular that the natural domain of definition of these operators is {\it strictly smaller} than what standard pluripotential theory would predict (see Section 5 for the details). We guess that this observation, rather obvious in the case of smooth functions, will play an important role in the resolution of the issue caused by the division of measures.

The note is organized as follows: in the next section we collect the basic notions from linear algebra, viscosity and pluripotential theory. Then we investigate the various notions of supersolutions in \cite{EGZ} and \cite{Lu} and compare them with the complex analogue of Trudinger's supersolutions. Section 3 is devoted to the proof of a very general comparison principle. Then in Section 4 we restrict our attention to operators depending on the eigenvalues of the complex Hessian matrix of the unknown function. We show existence and uniqueness of viscosity solutions under fairly mild conditions. One subsection is devoted to the regularity of these weak solutions. Using classical methods due to Walsh \cite{Wa} (see also \cite{BT1}) we show the optimal H\"older regularity for sufficiently regular data. Secton 5 is devoted to comparisons between viscosity and pluripotential subsolutions and supersolutions. Finally in Section \ref{lagrangian_section}, we solve the  Dirichlet problem for the Lagrangian phase operator. 

\medskip
\noindent
{\bf Acknowledgements}. We are grateful to   Vincent Guedj, Duong H. Phong and Ahmed Zeriahi for useful discussions. This work  has been partially written during the first-named author stay at the Institut de Mathematiques de Toulouse granted by prix Szolem Mandelbrojt and the third author visits to the Department of Mathematics of Columbia University and the Institute of  Mathematics of Jagiellonian University funded by ATUPS and EDMITT travel grants. We  would like to thank all these institutions for their hospitality. 
\section{Preliminaries}
In this section we collect the notation and the basic results and definitions that will be used throughout the note.
\subsection{Linear algebra toolkit}
We begin by introducing the notion of an admissible cone that will be used throughout the note:
\begin{definition}\label{admissiblecone}
A cone $\G$ in $\mathbb R^n$ with vertex at the origin is called admissible if:
\begin{enumerate}
\item $\G$ is open and convex, $\G\neq\R^n$;
\item $\G$ is symmetric i.e. if $x=(x_1,\cdots,x_n)\in\G$ then for any permutation of indices $i=(i_1,\cdots,i_n)$ the vector $(x_{i_1},\cdots,x_{i_n})$ also belongs to $\G$;
\item $\G_n\subset\G $, where $\G_n:=\lbr x\in\mathbb R^n|\ x_i> 0,\ i\in 1,\cdots, n\rbr$. 
\end{enumerate}
\end{definition}
From the very definition it follows that  $\G_n$ is an admissible cone. Other examples involve the  $\G_k$ cones that we describe below:

Consider the $m$-th elementary symmetric polynomial defined by
$$
\sigma_m (x)=\sum _{1\leq j_1 < ... < j_m \leq n }x _{j_1}x _{j_2} 
... x _{j_m}.
$$
We shall use also the normalized version
$$
S_m(x):= \binom{n}{m}^{-1} \sigma_{m}.
$$

\begin{definition}
For any $m=1,\ldots n$, the  positive cone $\Gamma_m$ of vectors $x=(x_1,\cdots,x_n)\in\mathbb R^n$ is defined by
\begin{equation}\label{ga}
\Gamma_m=\lbrace x\in\mathbb R^n|\ \sigma_1(x)> 0,\ \cdots,\
\sigma_m(x)> 0\rbrace.
\end{equation}
\end{definition}
It is obvious that these cones are open and symmetric with respect to a permutation of the $x_i$'s. It is a nontrivial but classical fact that $\G_m$ is also convex.

Exploiting the symmetry of $\G$ it is possible to discuss $\G$ positivity for Hermitian matrices:
\begin{definition} A Hermitian $n\times n$ matrix $A$ is called $\G$ positive (respectively $\G$-semi positive) if the vector of eigenvalues $\la(A):=(\la_1(A),\cdots,\la_n(A))$ belongs to $\G$ (resp. to the Euclidean closure $\bar{\G}$ of $\G$). The definition is independent of the ordering of the eigenvalues.
\end{definition}

Finally one can define, following \cite{Li}, the notion of $\G$-admissible and $\G$-subharmonic functions through the following definitions:
\begin{definition} A $C^2$ function $u$ defined on a domain $\Om\subset\C$ is called $\G$-admissible if for any $z\in\Om$ the complex Hessian $Hu(z):=[\frac{\partial^2}{\partial z_j\partial\bar{z}_k}]_{j,k=1}^n$ is $\G$-positive.
\end{definition}

  In particular, if $\Gamma$ is an admissible cone, then  $\Gamma\subset \Gamma_1$ (see \cite{CNS}), hence we have the following corollary:

\begin{corollary} Any $\G$-admissible function is subharmonic.
\end{corollary}

\begin{definition} An upper semicontinuous function $v$ defined on a domain $\Om\subset\C^n$ is called $\G$-subharmonic if near any $z\in\Om$ it can be written as a decreasing limit of local $\G$-admissible functions.
\end{definition}

We refer to \cite{HL} for a detailed discussion and potential theoretic properties of general $\G$-subharmonic functions.
\subsection{Viscosity sub(super)-solutions}
Let $\Omega$ be a bounded domain in $\C^n$. Consider the following equation:
\begin{equation}\label{general eq}
F[u]:= F(x,u, D u,Hu)=0, \quad {\rm on}\, \,\Omega,
\end{equation} 
where $Du=(\partial_{z_1} u,\ldots ,\partial_{z_n} u )$, $Hu=(u_{j\bar k})$ is the Hessian matrix of $u$ and $F$ is continuous on  $\Omega\times \R\times \C^n\times  \mathcal{H}^n$.
The operator $F$ is called {\sl degenerate elliptic} at a point $(z,s,p,M)$ if
\begin{equation}
F(z,s,p,M+N)\geq F(z,s,p,M) \quad \text{for all}\quad N\geq 0, N\in \mathcal{H}^n,
\end{equation}
where $\mathcal{H}^n$ is the set of Hermitian matrices of size $n\times n$.  We remark that in our case $F(z,s,p,M)$ is not neccesarily degenerate elliptic everywhere on $\Omega\times \R\times \C^n\times  \mathcal{H}^n$. Motivated by the paper of Trudinger \cite{Tr90}  we pose the following definition:
\begin{definition}\label{def}
A function $u\in L^\infty(\Omega)$  is a viscosity subsolution of (\ref{general eq}) if it is upper semi-continuous in $\Omega$ and for any $z_0\in \Omega$, and any $C^2$ smooth function $q$ defined in some neighbourhood of $z_0$ and satisfying $u\leq q$, $u(z_0)=q (z_0)$, the inequality 
\begin{equation}F[q](z_0)\geq 0
\end{equation}
holds.  We also say that $F[u]\geq 0$ in the viscosity sense and $q$ is an upper (differential) test for $u$ at $z_0$.

\medskip
A  function $v\in L^\infty (\Omega)$ is a viscosity supersolution of equation (\ref{general eq}) if it is lower semi-continuous and there are no points $z_0\in \Omega$ and functions  $C^2$ smooth functions defined locally around $z_0$, such that $v\geq q$ in $\Omega$, $v(z_0)=q(z_0)$ and 
\begin{equation}
\inf_{N\geq 0} F(z_0, q(z_0), D q (z_0), N+ H \psi (z_0) )>0.
\end{equation}
We also say that $F[u]\leq 0$ in the viscosity sense and $q$ is a lower (differential) test for $u$ at $z_0$.
\end{definition}

For fixed $(z,s,p)\in\Om\times\mathbb R\times \mathbb C^n$ the set of all Hermitian matrices $M$, such that $F$ is degenerate elliptic at $(z,s,p,M)$ is called the {\it ellipticity set} $\mathcal A(z,s,p)$ for the data $(z,s,p)$. Note that the ellipticity set has the property that
$$\mathcal A(z,s,p)+\G_n\subset\mathcal A(z,s,p),$$
but it may not be a cone. Throughout the note we shall however focus on the situation when the ellipticity set is a cone which is  moreover constant for all the possible data sets. We then define the {\it ellipticity cone} associated to the operator $F$ which is modelled on the notion of a subequation coined by Harvey and Lawson in \cite{HL} :
\begin{definition} An operator $F(z,s,p,M)$ has an ellipticity cone $\G$ if for any $M$ in the ellipticity set the vector $\la(M)$ of the eigenvalues of $M$ belongs to the closure $\bar{\G}$ of $\G$. Furthermore $\G$ is the minimal cone with such properties.
\end{definition}
Throughout the note we consider only the situation when $\G$  is an admissible cone in the sense of Definition \ref{admissiblecone}. We shall make also the following additional assumption (compare with the Condition \ref{condition2} in Subsection 4.1):
\begin{equation}\label{aaaa}
\forall \la\in\partial\G,\ \forall (z,\ s,\ p)\in\Omega\times\R\times\C^n\ \  F(z,s,p,\la)\leq0.
\end{equation}
This condition arises naturally whenever one seeks solutions to 
$$F(z,u(z), Du(z), Hu(z))=0$$
 with pointwise Hessian eigenvalues in $\G$ (recall that $F$ increases in the $\G_n$ directions).

It is evident that in Definition \ref{def} the notion of a supersolution is different and substantially more difficult that the notion of a subsolution. The reason for this is that there is no analog for the role of the positive cone $\G_n$ from the case of subsolutions in the supersolutions' case. As an illustration we recall that  while any plurisubharmonic function is a subsolution for $F(u):=\det(H(u))=0$ (see \cite{EGZ}) it is far from being true that all supersolutions can be written as the negative of a plurisubharmonic function.

\medskip
Below we also give another  notion of a supersolution that was coined in \cite{EGZ} for the Monge-Amp\`ere equation (see also \cite{Lu} for the case of $m$-Hessian operator). It can be generalized for all operators admitting an elliptic admissible cone:
\begin{definition}\label{def2} A lower semicontinuous function $u$ is said to be a supersolution for the operator $F(z,s,p,M)$ with the associated ellipticity cone $\Gamma$ iff for any $z_0\in\Om$ and every lower differential test $q$ at $z_0$ for which $\lambda( Hq(z_0))\in \bar{\Gamma}$ one has
$$F(z,q(z_0),Dq(z_0),Hq(z_0))\leq 0.$$
\end{definition}
Note that in the definition we limit the differential tests only to those for which  $\lambda(Hq(z_0))\in \bar{\Gamma}$ .

The next proposition shows that under the assumption (\ref{aaaa}) the definition above coincides with the one from Definition \ref{def}.
\begin{proposition} Suppose that the operator $F(z,s,p,M)$ satisfies (\ref{aaaa}). Then a lower semicontinous function $u$ defined on a domain $\Om$ is a supersolution for $ $F(z,s,p,M)$=0$ in the sense of Definition \ref{def2} if and only if it is a supersolution in the sense of Definition \ref{def}.
\end{proposition}
\begin{proof} Suppose first that $u$ is a supersolution in the sense of Definition \ref{def2}. Fix any $z_0$ in $\Om$ and $q$ a lower differential test for $u$ at $z_0$. If $ \lambda( Hq(z_0))\in \G$ then $$F(z,q(z_0),Dq(z_0),Hq(z_0))\leq 0,$$ hence taking $N=0$ in Definition \ref{def} we see that the condition is fulfilled. If $\lambda (Hq(z_0))$ fails to be in $\G$ then there is a positive definite matrix $N$ and a positive number $t$ such that $\lambda( Hq(z_0)+tN)\in\pa \G$. But this implies that $F(z,q(z_0),Dq(z_0),Hq(z_0)+tN)\leq 0$ which fulfills the condition in Definition \ref{def} again.

Suppose now that  $u$ is a supersolution in the sense of Definition \ref{def}. Again choose  $z_0$ in $\Om$ and $q$ a lower differential test for $u$ at $z_0$. We can assume that $\lambda(Hq(z_0))$ is in $\overline{\G}$, for otherwise such a differential test cannot be apllied in Definition \ref{def2}. But then by ellipticity
$$ F(z,q(z_0),Dq(z_0),Hq(z_0))\leq  F(z,q(z_0),Dq(z_0),Hq(z_0)+N), \,\forall  N\geq 0, N\in \mathcal{H}^n.$$
The infimum over $N$ for the right hand side is non positive by definition which implies
$$F(z,q(z_0),Dq(z_0),Hq(z_0))\leq 0$$
which was to be proved.
\end{proof}

\subsection{Aleksandrov-Bakelman-Pucci maximum principle}

%
%
In this section, we recall a variant of Aleksandrov-Bakelman-Pucci (ABP) maximum principle following \cite{Jen}. We first recall the following definition (cf. \cite{Jen}): 
\begin{definition}
Let $\Omega$ be a bounded domain in $\R^n$ centered at the origion, and $u\in C(\overline \Omega)$. We define 
$$E_\delta=\{x\in \Omega| \text{ for some } p\in \overline{B(0,\delta)}, u(z)\leq u(x)+p.(z-x),\forall z\in \Omega \}.$$
\end{definition}

\noindent
Then we have the following lemma due to Jensen \cite{Jen} which will be used in the proof of Lemma \ref{general_comparision}. Recall that a function $u$ is said to be semi-convex if $u+k|z|^2$ is convex for a sufficiently large constant $k$.
\begin{lemma}
\label{ABP}
Let $u\in C(\overline \Omega)$ be semi-convex for some constant $k>0$. If $u$ has an interior maximum and $\sup_\Omega u -\sup_{\partial\Omega } u=\delta_0 d  >0 $, where $d=diam(\Omega)$. Then there is a constant $C=C(n,k)>0$ such that 
\begin{equation}
|E_\delta |\geq C\delta^n, \quad \text{for all }\delta \in (0,\delta_0).
\end{equation}
\end{lemma}
\begin{proof} 
As in Jensen \cite{Jen}, by regularization,  we can reduce to the case when $u\in C^2(\Omega)$.  Now, suppose  that $u$ has an interior maximum at $x_0$ and  
$$\delta_0= \frac{\sup_\Omega u -\sup_{\partial \Omega} u}{d}=\frac{u(x_0) -\sup_{\partial \Omega}u }{d},$$
where $d=diam(\Omega)$. 

\medskip
 We now prove that for $\delta< \delta_0$  we have $ B(0,\delta)\subset Du(E_\delta)  $. Indeed, for any $p\in B(0,\delta)$, consider the hyperplane $\ell_p(x)=h+\langle p, x \rangle$ where $h=\sup_{y\in \Omega}( u(y)-\langle  p,y \rangle)$. Then we have $u(x)\leq \ell_p(x)$ on $\Omega$ and $ u(x_1)=\ell_p(x_1)$ for some  $x_1\in \overline{\Omega}$. If we can  prove that $x_1\in \Omega$, then $Du(x_1)= p$, so $ B(0,\delta)\subset Du(E_\delta)  $. Suppose by contradiction that $x_1\in \partial \Omega$, then
\begin{eqnarray*}
\sup_\Omega u&=&u(x_0)\\
&\leq &\ell_p(x_1) +\langle p, x_0-x_1\rangle\\
&=& u(x_1)+\langle p, x_0-x_1\rangle\leq \sup_{\partial \Omega} u + \delta d< \sup_{\partial \Omega} u+ \delta_0 d=\sup_\Omega u,
\end{eqnarray*}
hence we get a contradiction. 

\medskip
 Next, as we have proved that $ B(0,\delta)\subset Du(E_\delta)  $, then by comparing volumes, we infer that 
\begin{equation}
c(n)\delta^n\leq \int_{E_\delta} |\det (D^2 u)|.
\end{equation}
Since $u$ is semi-convex with the constant $k>0$ and $D^2u\leq 0$  in $E_\delta$, we have $|\det (D^2 u)|\leq k^n $. It follows that 
$|E_\delta| \geq c(n) k^{-n}\delta^n$.
\end{proof}

\subsection{$\Gamma$-subharmonic functions}
We have defined $\G$ subharmonic functions as limits of admissible ones. Below we present the alternative viscosity and pluripotential points of view:

Let $\Omega\subset \C^n$ be a bounded domain. Denote $\omega=dd^c|z|^2$, where  $d:=i(\bar \partial+\partial)$ and $d^c:=\frac{i}{2\pi}(\bar \partial-\partial )$  so that $dd^c=\frac{i}{\pi}\partial \bar{\partial}$. Let  $\Gamma \subsetneq \R^n$ be an admissible cone as in  Definition \ref{admissiblecone}.  
We first recall the definition of $k$-subharmonic function:
\begin{definition}
We call a function $u\in C^2(\Omega)$ is  $k$-subharmonic  if for any $z\in \Omega$, the Hessian matrix $(u_{i \bar j })$ has eigenvalues forming a vector in the closure of the cone $\Gamma_k$. 
\end{definition}
Following the ideas of Bedford-Taylor \cite{BT2}, Blocki \cite{Blo05} introduced the pluripotential definition of the $k$-sh function. 

\begin{definition}
Let $u$ be subharmonic function on a domain $\Omega\subset \C^n$. Then $u$ is called $k$-subharmonic ($k$-sh for short) if for any collection of $C^2$-smooth $k$-sh functions $v_1,\ldots,v_{k-1}$, the inequality 
$$dd^cu\wedge dd^c v_1\wedge\ldots\wedge dd^cv_{k-1}\wedge \omega^{n-k}\geq 0$$ 
holds in the weak sense of currents.
\end{definition}

\noindent
For a general cone $\Gamma$, we have the following  definition in the spirit of viscosity theory:
\begin{definition}
An upper semicontinuous function  $u$  is called  $\Gamma$-subharmonic  (resp. strictly  $\Gamma$-subharmonic) if for any $z\in \Omega$,  and any upper test function $q$ of $u$ at $z$, we have 
$$
 \lambda(Hq (z))\in \overline{\Gamma}\quad (resp. \, \lambda(Hq (z))\in\Gamma).
$$ 
\end{definition}

By definition,  if $u$ is a  $ \Gamma$-subharmonic function, it is a $\Gamma$-subsolution in the sense of  Sz\'ekelyhidi \cite{Sze1}. In particular,  when $\Gamma=\Gamma_k$ for $k=1,\ldots n$, $u$ is a viscosity subsolution of the equation $$S_{k}(\lambda(Hu))=0,$$
where
$$ S_k(\lambda (Hu))={(dd^c u)^k \wedge \omega^{n-k}\over \omega^n}.$$
Then it follows from \cite{EGZ, Lu} that $u$ is a $k$-subharmonic function on $\Omega$, hence $u$ is a subharmonic function if $k=1$ and a plurisubharmonic function if $k=n$. 

\medskip
We also have the following definition generalizing the pseudoconvex domains (see also \cite{Li} for similar definition for smooth domains):
\begin{definition}
Let $\Omega$ be a  bounded domain in $\C^n$, we say that $\Omega$ is a $\Gamma$-pseudoconvex domain if there is a constant $C_\Omega>0$ depending only on $\Omega$ so that $-d(z)+C_\Omega d^2(z)$  is $\Gamma$-subharmonic on $\partial \Omega$, where $d(z):=dist(z,\partial \Omega)$. 
\end{definition}
\noindent
We recall the following lemma which was proved in \cite[Theorem 3.1]{Li}.
\begin{lemma}
\label{defining_function} Let $\Omega$ be bounded domain in $\C^n$ with $C^2$ smooth boundary.
 Let $\rho\in C^{2}(\bar \Omega)$ be a defining function of $\Omega$ so that $\lambda(H\rho)\in \Gamma$ on $ \partial \Omega$. Then there exists a defining function $\tilde{\rho}\in C^{2 }(\bar \Omega)$  for $\Omega$ such that $\lambda (H \tilde{ \rho}) \in \Gamma$ on $ \overline \Omega$. 
\end{lemma}

Finally we wish to recall the survey article \cite{Ze} where the Reader may find a thorough discussion of the viscosity theory associated to complex Monge-Amp\`ere type equations.

 \section{Comparison principles}
Comparison principles are basic tools in pluripotential theory- we refer to \cite{K98, GZ2} for a thorough discussion of these inequalities. In viscosity theory one compares sub- and supersolutions to the same equation. It is a crucial observation (cf.  \cite{EGZ}) that even though supersolutions may fail to have nice pluripotential properties a version of the comparison principle holds for the complex Monge-Amp\`ere equation. In this section we discuss under what assumptions such comparison principles hold for general operators.

 \subsection{A preliminary comparison principle}
Let $\Omega$ be a bounded domain in $\C^n$. In this subsection we prove a comparison principle for viscosity solutions of the following equation:
\begin{equation}
\label{general_eq2}
F[u]:= F(x,u, D u,Hu)=0.
\end{equation} 
It is well known that mere ellipticity is insuffiecient to guarantee comparison type result. Hence we add some natural structural conditions for the equation \ref{general_eq2}.

First of all we  assume  that $F$ is decreasing in the $s$ variable, namely
\begin{equation}\label{proper}
\forall r>0\ F(z,s,p,M)-F(z, s+r, p,M)\geq 0.
\end{equation}
This is a natural assumption in the theory (see \cite{Ze}) as it yields an inequality in the "right" direction for the maximum principle.

Next we assume certain continuity property with respect to the $z$ and $p$ variables: 
\begin{equation}\label{lipschizt}
|F(z_1,s,p_1,M)-F(z_2, s,p_2,M )|\leq \alpha_z(|z_1-z_2|)+\alpha_p(|p_1-p_2|) ,
\end{equation}
for all $z_1,\ z_2\in \Omega$, $\s\in\R$,  $p_1,\ p_2\in \C^n$, $M\in \mathcal{H}^n$. Here $\alpha_z$ and $\alpha_p$ are certain moduli of continuity i.e. increasing functions defined for nonnegative reals which tend to zero as the parameter decreases to zero.

We can now state the following general comparison principle for the equation (\ref{general_eq2}).

\begin{lemma}\label{general_comparision}
Suppose $u\in L^\infty(\overline \Omega)$ (resp. $v\in L^\infty (\overline{\Omega})$) satisfies $F[u]\geq \delta$ (resp. $F[v]\leq 0$) in $\Omega$ in the viscosity sense for some $\delta>0$. Then 
\begin{equation}\label{comparison1}
\sup_{\Omega}(u-v)\leq \max_{\partial \Omega} \left\lbrace  (u-v)^*,0\right\rbrace,
\end{equation}
with $^*$ denoting the standard upper semicontinuous regularization.
\end{lemma}
\begin{proof}

The idea comes from \cite{Tr90}. 
We use Jensen's approximation (cf. \cite{Jen}) for $u,v$ which is defined by 
\begin{eqnarray} \nonumber
u^\e(z)&=& \sup_{z'\in \Omega}\left\lbrace u(z')- \frac{C_0}{\e} |z'-z|^2\right\rbrace,\\\label{Jensen_appr}
v_\e(z)&=&\inf_{z'\in \Omega}\left\lbrace v(z')+\frac{C_0}{\e} |z'-z|^2\right\rbrace,
\end{eqnarray}
where $\e>0$ and $C_0=\max\{{\rm osc}_\Omega u, {\rm osc}_\Omega v\}$  with ${\rm osc}(u)=\sup u_\Omega-\inf_\Omega u$. Then the supermum and infimum in  (\ref{Jensen_appr}) are achieved at points $z^*, z_*\in \Omega$ with $|z-z^*|, | z-z_*|<\e $ provided that $z\in \Omega_\e=\{z\in \Omega| \, dist(z,\partial \Omega)>\e \}$.  
 It follows from \cite{CC95} (see also  \cite{Wa12} for an adaption in the complex case) that $u^\e$ (resp. $v_\e$) is Lipschitz  and semi-convex (resp. semi-concave) in $\Omega_\e$, with 
 \begin{equation}
 |D u^\e|, |D v_\e|\leq \frac{2C_0}{\e}, \quad H u^\e, -H v_\e \geq -\frac{2C_0}{\e^2}{\rm Id} ,
 \end{equation}
whenever these derivatives are well defined.

Exploiting the definition of viscosity subsolution one can show that $u^{\e}$ satisfies 
\begin{equation}
\label{visco_e}
F(z^*, u^\e(z), D u^\e(z), H u^\e(z))\geq \delta
\end{equation}
in the viscosity sense for all $z\in \Omega_\e$.
 Indeed, let $q$ be an upper test of $ u^\e$ at $z_0$, then the function
$$\tilde{q}(z):= q( z+z_0-z_0^*) +\frac{1}{\e}|z_0-z_0^* |^2$$ 
is an upper test for $u$ at $z_0^*$.  Therefore we get (\ref{visco_e}) as $u$ is a viscosity subsolution. This also implies that 
\begin{equation}
F(z^*, u^\e(z), D u^\e(z), N+ H u^\e(z))\geq \delta,
\end{equation}
in the viscosity sense for any fixed matrix $N\geq 0$.  Since  any locally semi-convex (semi-concave) function is twice differentiable  almost everywhere by Aleksandroff's theorem, we infer that for almost all $z\in \Omega_\e$, $F$ is degenerate elliptic at $(z^*, u^\e(z), D u^\e(z), H u^\e(z))$ and
\begin{equation}\label{visco_ineq1}
F(z^*, u^\e(z), D u^\e(z), N+ H u^\e(z))\geq \delta  ,
\end{equation}
for all $N\in \mathcal{H}^n$ such that $N\geq 0$. 

\medskip

We assume  by contradiction that $\sup_\Omega (u-v) =u(z_0) -v( z_0)=a>0$ for some $z_0\in  \Omega$.  For any $\e$ sufficiently small the function $w_\e:= u^\e- v_\e$ has a positive maximum on $\Omega_\e$ at some point $z_\e\in \Omega_\e$ such that $z_\e \rightarrow z_0$ as $\e\rightarrow 0$.   So we can choose $\e_0>0$ such that that for any $\e<\e_0$,  $w_\e:= u^\e- v_\e$ has a positive maximum on $\Omega_\e$ at some point $z_\e\in \Omega$ with $d(z_\e,\partial \Omega)>\e_0$. 
Applying the ABP maximum principle (Lemma \ref{ABP}), for  the function $w_\e$ on $\Omega_{\e_0} $ and for any  $\lambda>0$ sufficiently small, there exist a set $E_\lambda \subset \Omega_{\e_0}$ containing $z_\e$ with $| E_\lambda|\geq c\lambda^n$, where  $c$ is $c(n)\e^{2n} $,  such that $|D w_\e|\leq \lambda $ and $H w_\e\leq 0$ almost everywhere in $E_\lambda$.  Since $w_\e(z_\e)>0$, we can choose $\lambda$ small enough such that $w_\e\geq 0$ in $E_\lambda$.  The condition (\ref{proper}) and the fact that $F$ is degenerate elliptic at $(z^*, u^\e(z), D u^\e(z), H u^\e(z))$ for almost all $z\in E_{\lambda}$,  imply that 
\begin{equation}
\label{ineq_comparison}
F(z^*, u^\e(z), D u^\e(z), N+ H u^\e(z))\leq F(z^*, v_\e(z), D u^\e(z), N+ H v^\e(z)).
\end{equation}
Using (\ref{lipschizt}) and the fact that $|D(u^\e-v_\e) |\leq \lambda$, we get
$$F(z^*, v_\e(z), D u^\e(z), N+ H v^\e(z))\leq  F(z^*, v_\e(z), D v^\e(z), N+ H v^\e(z)) +\alpha_p(\lambda).$$
Combining with (\ref{lipschizt}), (\ref{visco_ineq1}), (\ref{ineq_comparison})  and $|z^*-z_*|<\e$ that for almost all $z\in  E_\lambda$, 
\begin{equation}
F(z_*,v_\e(z),D v_\e(z), N+Hv_\e(z))\geq \delta-\alpha_z(\e)-\alpha_p(\lambda).
\end{equation} 
By taking $\lambda$, and then $\e$ sufficiently small and using the  fact that $v_\e$ is twice differentiable almost everywhere on $\Omega$, we can find at a fixed point $z_1\in E_\lambda$ a lower test $q$ of $v$  at $z_1$ such that 
\begin{equation}
F\left(z_0, q (z_0), D(z_0), N+Hq(z_0)\right)\geq \frac{1}{2}\delta,
\end{equation}
for all $N\geq 0$. This contradicts the definition of viscosity supersolution. Therefore we  get (\ref{comparison1}).    
\end{proof}

\noindent
{\bf Remark.}
By assuming more properties of $F$, it is possible to obtain $\delta=0$ in the previous result. This is the case for the Monge-Amp\`ere equation. Otherwise we need to adjust  function $u$ to achieve a strict inequality in order to use Lemma \ref{general_comparision}.

\subsection{Comparison principle for Hessian type equations}
We now consider the Hessian type equation of the form
\begin{equation}\label{Hessian_equation}
F[u]=\psi(z,u),
\end{equation}
where $\psi \in C^0(\Omega\times \R)$ and  $ F[u]=f(\lambda(Hu))$ such that

\begin{equation}
 s\mapsto \psi (\cdot,s)\, \text{ is weakly increasing},
\end{equation}
\begin{equation}
f\in C^0(\bar \Gamma), f>0 \,\, {\rm on}\,\, \Gamma, f=0 \,\, {\rm on}\,\, \partial\Gamma,
\end{equation}
and
\begin{equation}
 f(\lambda+\mu)\geq f(\lambda),\, \,\forall \lambda\in \Gamma, \mu\in \Gamma_n.
\end{equation}

 First, {\it In order to use Lemma \ref{general_comparision}, we extend $f$ continuously on $\R^n$ by taking $f(\lambda)=0$ for all $\lambda\in \R^n\setminus \Gamma$. } 
\medskip
For a $\delta$ independent comparison principle we need more assumptions on $F$.  Similarly to  \cite{Tr90}, we can assume that the operator $F[u]=f(\lambda(Hu))$ satisfies 
\begin{equation}\label{conditioncomparison}
\sum_{i=1}^n\frac{\partial f}{\partial \lambda_i}\lambda_i=\sum_{i=1}^nf_i\lambda_i\geq \nu (f) \text{ in } \Gamma,\ inf_{z\in\Omega}\psi(z,\cdot)>0
\end{equation}
for some positive increasing function $\nu$.  

This condition is satisfied for example in the case of the complex Hessian equations $F[u]:=\sigma_k(\lambda(Hu))$, $k\in\lbrace1,\cdots,\ n\rbrace$.

We also study a  {\it new condition} namely
\begin{equation}\label{condition_comparison2} 
  f\ is\ concave\ and\ homogeneous,
\end{equation}
i.e $f(t\lambda)=t f(\lambda), \forall t\in \R^{+}$. 

\begin{theorem}\label{comprison_principle}
Let $u,v\in L^\infty(\overline{\Omega})$ be viscosity subsolution and supersolution of equation (\ref{Hessian_equation}) in $\Omega$.  Assume that either $f$ satisfies either  (\ref{conditioncomparison}) or (\ref{condition_comparison2}).  Then
\begin{equation}\label{comparison_2}
\sup_{\Omega}(u-v)\leq \max_{\partial \Omega} \left\lbrace  (u-v)^*,0\right\rbrace.
\end{equation}
\end{theorem}
\begin{proof}
Assume first that $f$ satisfies (\ref{conditioncomparison}). Then following \cite{Tr90}, we set for any $t\in (1,2)$,
$$u_t(z)=tu(z) -C(t-1) ,$$ where $C=\sup_\Omega u $. Therefore we have  $u_t(z)\leq u(z) $  
on  $\Omega$ for all $t\in (1,2)$. Then for any $z_0\in \Omega$ and an upper test function $q_t(z) $ of $u_t$ at $z_0$ we have $q(z):=t^{-1} q_t(z)-C(t^{-1}-1)$  is also an upper test for $u$  at $z_0$. Set $\lambda=\lambda[q](z_0)$, then $\lambda[q_t](z_0)=t\lambda$ and $q(z_0)\geq q_t(z_0)$.   We also recall that  the function $s\mapsto f(s\lambda)$ is increasing on $\R^+$ (by (\ref{conditioncomparison})) and $f(\lambda) \geq \psi(z,u(z_0))$ since $q$ is an upper test for $u$ at $z_0$.  It follows that  at $z_0$,
\begin{eqnarray*}
F[q_t]&=&f(\lambda[q_t] )=f(t\lambda)\\
&\geq & f(\lambda)+ (t-1) \sum \lambda_if_i(t^*\lambda)\\
&\geq &\psi (z_0, q(z_0))+ (t-1) \sum \lambda_if_i(t^*\lambda)\\
&\geq& \psi (z_0, q_t(z_0)) +\frac{t-1}{2}\nu(\inf_{\Omega} \psi(z,\inf_\Omega u))
\end{eqnarray*}
for $1\leq t^*<t$, sufficiently close to 1. Therefore we have for some $\delta>0$
$$F[u_t]\geq \psi (z,u_t)+\delta,$$ in the viscosity sense in $\Omega$. Thus the inequality  (\ref{comparison_2}) follows from Lemma \ref{general_comparision}. 

\medskip
Next, consider  the second case when $f$ is concave and homogeneous. Suppose, without loss of generality, that $0\in \Omega$.  We  set $$u_\tau(z)=u(z)+\tau  (|z|^2-R),$$
where $R=diam(\Omega)$. 
Then for any $q_\tau\in C^2(\Omega)$ such that $q_\tau \geq u_\tau$ near $z_0$ and $q_\tau(z_0)=u_\tau(z_0)$, we have $q=q_\tau-\tau(|z|^2-R) \geq q_\tau$, and $q$ is also an upper test for $ u$ at $z_0$. Therefore, we have at $z_0$,
\begin{eqnarray}
F[q_\tau]&= &2^df\left(\frac{\lambda(Hq)+\tau {\bf 1}}{2} \right) \\ \nonumber
&\geq & f(\lambda(Hq))+ f(\tau {\bf 1})\\ \nonumber
&\geq& \psi(z_0, q_\tau) +\delta.
\end{eqnarray}
Therefore $F[u_\tau]\geq \psi +\delta  $ in the viscosity sense.
Applying Lemma \ref{general_comparision} we get (\ref{comparison_2}). 
\end{proof}

By definition, we have the following  properties of sub(super)-solutions. Their proofs follow in a straightforward way from \cite[Proposition 4.3]{CIL}. 
\begin{lemma}
{\rm (a)} Let $\{u_j\}$  be viscosity subsolutions of (\ref{Hessian_equation}) in $\Omega$, which are uniformly bounded from above. Then $(\limsup_\Omega u_j)^*$ is also a viscosity subsolution of (\ref{Hessian_equation}) in $\Omega$.\\
{\rm (b) }Let $\{v_j\} $  be viscosity supersolutions of (\ref{Hessian_equation}) in $\Omega$, which are uniformly bounded from below. Then $(\liminf_\Omega  v_j)_*$ is also a viscosity supersolution of (\ref{Hessian_equation}) in $\Omega$.\\
\end{lemma}
\noindent
Now using  Perron's method (see for instance \cite{CIL}), we obtain the next result:
\begin{lemma}
\label{Perron}
Let $\underline{u}, \overline{u}\in L^{\infty}(\Omega)$ are a subsolution and a supersolution of (\ref{Hessian_equation}) on $\Omega$.  Suppose that  $\underline{u}_*(z)=\overline{u}^*(z)$  on the boundary of $\Omega$.  Then the function 
$$u:=\sup\{v\in L^{\infty}(\Omega)\cap USC(\Omega): v \text{ is a subsolution of } (\ref{Hessian_equation}), \underline{u}\leq v\leq \overline{u} \}$$ 
satisfies  $u\in C^0(\Omega)$ and
$$F[u]=\psi(x,u) \quad  \text{in } \Omega,$$
is the viscosity sense. 
\end{lemma}
\begin{proof}
It is straightforward  that $u^*$ is a viscosity subsolution of (\ref{Hessian_equation}).  We next prove that $u_*$ is a supersolution of (\ref{Hessian_equation}). Assume by contradiction that $u_* $ is not a supersolution of (\ref{Hessian_equation}), then there exists a point $z_0\in \Omega$ and a lower differential test $q$ for $u_*$ at $z_0$ such that
\begin{equation}
F[q](z_0) > \psi(z_0,q(z_0)).
\end{equation}
Set $\tilde q(z)=q(z)+b-a|z-z_0|^2$, where $b=(ar^2)/6$ with $a,r>0$ small enough so that $F[\tilde{q}]\geq \psi(x,\tilde{q})$ for all $|z-z_0|\leq r$. Since $u_*\geq q$ for $|z-z_0|\leq r$, we get $u^*\geq u_*>\tilde{q}$ for $r/2\leq |z-z_0|<r$. Then the function
$$w(z)=\begin{cases}
\max\{u^*(z),\tilde{q}(z)\} \text{ if } |z-z_0|\leq r,\\
 u^*(z)  \, \text{ otherwise}
\end{cases}$$
is a viscosity subsolution of (\ref{Hessian_equation}). By choosing a sequence $z_n\rightarrow z_0$ so that $u(z_n)\rightarrow u_*(z_0)$, we have $\tilde{q}(z_n)\rightarrow u_*(z_0)+b$.
Therefore, for $n$ sufficiently large, we have $w(z_n)>u(z_n)$ and this contradicts the definition of $u$. Thus we have $u_*$ is also a supersolution.  Then it follows from Theorem \ref{comprison_principle} and $\underline{u}_*(z)=\overline{u}^*(z)$  for $z\in \partial \Omega$ that  $u^*\leq u_* $ on $\Omega$, hence $u=u_*=u^*$.
\end{proof}

\section{Dirichlet problems}  
\subsection{Viscosity solutions in $\Gamma$-pseudoconvex domains}
\label{solve_1}
Let $\Omega\subset\C^n$ be a $C^2$ bounded domain. 
In this section, we study the following Dirichlet problem
\begin{equation}\label{Dirichlet_gama_pscv}
\begin{cases} F[u]=f(\lambda(Hu))=\psi(x,u) \text{ on } \Omega\\ 
 u=\f \text{ on } \partial \Omega,
 \end{cases}
\end{equation}
where $\varphi\in C^0(\partial \Omega)$ and $\psi \in C^0(\overline{\Omega}\times \R)$   such that $\psi>0$ and 

 $$ s\mapsto \psi(.,s)\, \text{ is weakly increasing}.$$ 
Let $\Gamma\subsetneq \R^n$ be an admissible cone. We assume further that $f \in C^0(\overline{\Gamma})$  satisfies:
\begin{enumerate}
\item $f$ is concave and $ f(\lambda+\mu)\geq f(\lambda),\, \,\forall \lambda\in \Gamma, \mu\in \Gamma_n$.\label{condition 1}
\item $\sup_{\partial \Gamma} f=0$, and $f>0$ in $\Gamma$. \label{condition2}

\item $f$ is homogeneous on $\G$.  \label{condition 4}
 \end{enumerate}
We remark that, the condition (\ref{condition2}) and (\ref{condition 4}) imply that for any   $\lambda \in \Gamma$ we have 
\begin{equation}
\label{condition 3}
\lim_{t\rightarrow \infty}f(t\lambda)=+\infty.
\end{equation}
We now can solve the equation (\ref{Dirichlet_gama_pscv}) in the viscosity sense:
\begin{theorem}\label{existence_general}
Let $\Omega$ be a $C^2$ bounded $\Gamma$-pseudoconvex domain in $\C^n$. The the Dirichlet problem 
$$f(\lambda[u])=\psi (x,u) \text{  in } \Omega,\quad u=\f  \text{ on } \partial \Omega.$$
admits a unique admissible solution $u\in C^0(\bar \Omega) $.

\medskip
 In particular, we have a $L^\infty$ bound for $u$ which only depends on $||\varphi||_{L^
\infty}$ and $ ||\psi(x,C)||_{L^\infty}$ and $\Omega$, where $C$ is a constant depending on $\Omega$. 
\end{theorem}

\begin{proof}
By Lemma \ref{defining_function}, there is a  defining  function $\rho\in C^2(\overline{\Omega})$  for  $\Omega$ such that $\lambda(H\rho)\in \Gamma$ on $\overline \Omega$. The $C^2$-smoothness of the boundary implies the existence of a harmonic function  $h$ on $\Omega$ for arbitrary given continuous boundary data $\f$. Set $$\underline{u}=(A_1\rho+ h) + A_2\rho,$$  where $A_1>0$ is chosen so that $A_1\rho+ h$ is admissible and $A_2$ will be chosen later.

\medskip
By the concavity of $f$ and (\ref{condition 3}), for $A_2$ sufficiently large we get
\begin{align*}
f(\lambda [\underline u])&\geq& \frac{1}{2} f(2\lambda[A_1\rho+h]) +\frac{1}{2} f(2A_2\lambda[\rho])\\ 
&\geq&  \max_{\overline \Omega}\psi (x,h)\geq \psi (x,\underline u).
\end{align*}
Therefore $\underline{u}$ is a subsolution of 
 (\ref{Dirichlet_gama_pscv}). 

\medskip
Since $h$ is harmonic, for each $z\in\Omega$ there is  a Hermitian matrix $N\geq 0$ so that $\lambda( N+H(h)(z))\in\partial \Gamma$. But then then $f(\lambda(N+H(h)(z) ))=0$. Therefore, $ \bar v:=h$ is a supersolution of (\ref{Dirichlet_gama_pscv}).

\medskip 
Finally, the existence of solution follows from Perron's method. We set
$$u:=\sup\{w \text{ is subsolution of } (\ref{Dirichlet_gama_pscv}) \text{ on } \Omega, \underline{u}\leq w\leq \overline{v} \}.$$
As in the argument from Lemma \ref{Perron} we have $u^*$ (resp. $u_*$) is a subsolution (resp. supersolution) of (\ref{Dirichlet_gama_pscv}). It follows from the comparison principle (Theorem \ref{comprison_principle}) that   
$$u^*(z)-u_*(z)\leq \limsup_{w\rightarrow \partial \Omega} (u^*-u_*)^+(w).$$
Since  $\underline{u}$ and $\overline{v}$ are continuous and $\underline{u}=\overline{v}=\varphi$ on $\partial \Omega$ we infer that $u^*\leq u_*$ on $\Omega$ and $u^*=u_*$ on $\partial \Omega$. 
Therefore $u=u^*=u_*$ is a viscosity solution of (\ref{Dirichlet_gama_pscv}). The uniqueness follows from the comparison principle (Theorem \ref{comprison_principle}). 
\end{proof}

 As a corollary of Theorem \ref{existence_general},  we  solve the following Dirichlet problem for Hessian quotient equations
\begin{equation}\label{Hessian_quotient1}
\begin{cases}
S_{k,\ell}(\lambda(Hu)):=\frac{S_k}{S_\ell}(\lambda(Hu))=\psi(x,u)\quad \text{on } \Omega\\
u=\f \quad \text{on } \partial \Omega 
\end{cases},
\end{equation}
where  $\Omega\subset \C^n$ be a smooth  bounded $\Gamma_k$-pseudoconvex domain, $1\leq \ell<k\leq n$ and 
$$ S_k(\lambda (Hu))={(dd^c u)^k \wedge \omega^{n-k}\over \omega^n}.$$

Note that the operator $S_{k,\ell}^{1/(k-l)}$ is concave and homogeneous (see \cite{Sp}).

\begin{corollary}
The Dirichlet problem (\ref{Hessian_quotient1}) admits a unique viscosity solution $u\in C^0(\bar \Omega)$ for any continuous data $\f$. 
\end{corollary}

We also remark that a  viscosity subsolution  is always a $\Gamma$-subharmonic function. 

\begin{lemma}\label{viscosity_pluripotential}
Any viscosity subsolution of the equation
$f(\lambda(Hu))=\psi(z,u)$
is a $\Gamma$-subharmonic function.  In particular, if $u$ is a
viscosity subsolution of the equation
\begin{equation}
S_{k,\ell}(\lambda(Hu))=\psi(z,u),
\end{equation} 
then $u$  is $k$-subharmonic.
\end{lemma}

\begin{proof}
Let $z_0\in \Omega$ and $q\in C_{loc}^2(\{z_0\})$, such that $u -q$ attains its maximum at $z_0$ and $u(z_0)= q(z_0)$.  By definition we have
$$f(\lambda(Hq)(z_0))> 0. $$
Observe that for any semi-positive Hermitian maxtrix $N$, the function
$$q_N(z):=q(z) + \left\langle  N(z-z_0),z-z_0  \right\rangle$$
is also an upper test function for $ u$ at $z_0$. By the definition of viscosity subsolutions we have
\begin{equation}
\label{q_tilde}
f(\lambda(H\tilde{q} )(z_0))> 0.
\end{equation}
Suppose that  $\lambda(H q)(z_0)\notin \bar \Gamma$. Then we can find $N\geq 0$ so that  $\lambda(H\tilde{q})(z_0)\in \partial\Gamma$, so $f(\lambda(H\tilde{q})(z_0))=0$ by the condition (3) above, this contradicts to (\ref{q_tilde}). Hence we always have $\lambda[q] (z_0) \geq 0$, and so $u$ is $\Gamma$-subharmonic. 
\end{proof}

\subsection{H\"older continuity  of Hessian type  equations}
In this subsection, we study the H\"older continuity of the viscosity solution obtained in Section \ref{solve_1} to  the Dirichlet problem \begin{equation}\label{Dirichlet_gama_pscv_2}
\begin{cases} F[u]=f(\lambda(Hu))=\psi(x,u) \text{ on } \Omega\\ 
 u=\f \text{ on } \partial \Omega,
 \end{cases}
\end{equation}
where $f,\varphi$ and $\psi$ satisfy the  conditions spelled out in the previous subsection. 
We prove the following result:
\begin{theorem} Let $\Om$ be a strictly $\G$ pseudoconvex domain. Let $u$ be the viscosity solution of (\ref{Dirichlet_gama_pscv_2}). Suppose that $\varphi \in C^{2\al}({\pa\Om})$ for some $\al\in(0,1)$. If additionally $\psi (z,s)$ satisfies
\begin{enumerate}
\item $|\psi(z,s)|\leq M_1(s)$ for some $L^{\infty}_{loc}$ function $M_1$;
\item $|\psi(z,s)-\psi(w,s)|\leq M_2(s)|z-w|^{\al}$ for some $L^{\infty}_{loc}$ function $M_2$;
\end{enumerate} 
Then $u\in C^{\al}({\overline\Om})$. 
\end{theorem}

{\bf Remark}. Classical examples (see \cite{BT1}) show that the claimed regularity cannot be improved. Conditions 1 and 2 can be regarded as a weak growth conditions and seem to be optimal. If $\psi$ does not depend on the second variable then these conditions mean that $\psi$ is globally bounded and contained in $C^{\al}$.
\begin{proof}
The proof relies on the classical idea of Walsh-\cite{Wa}. Similar agrument was used by Bedford and Taylor- \cite{BT1} who dealt with the complex Monge-Amp\`ere operator. We shall apply a small adjustment in the construction of the local barriers which is due to Charabati \cite{Cha}.

Suppose for definiteness that $0\in\Om$.
Assume without loss of generality that  the $\G$-subharmonic function $\rho=-dist(z,\pa\Om)+C_{\Om}dist(z,\pa\Om)^2$ satisfies $F(\rho)\geq2$ (multiply $\rho$ by a constant if necessary and exploit the homogeneity of $F$). Recall $\rho$ vanishes on $\pa\Om$. As $\pa\Om\in C^2$ we know that $\rho\in C^{2}$ near the boundary. Then it is easy to find a continuation of $\rho$  in the interior of $\Omega$ (still denoted by $\rho$), so that $\rho$ is  $\G$-subharmonic and satisfies $F(\rho)\geq1$. 

Fix $\xi\in\pa\Om$. There is a uniform $C>>1$ (dependent on $\Om$, but independent on $\xi$) such that  the function
$$\gk(z):=C\rho(z)-|z-\xi|^2$$
is $\G$-sh. In particular $\gk\leq 0$ in $\overline{\Om}$.

By definition there is a constant $\tilde{C}$, such that for any $z\in\pa\Om$
$$\varphi(z)\geq \varphi(\xi)-\tilde{C}|z-\xi|^{2\al}.$$
 Consider the function $h_{\xi}(z):=-\tilde{C}(-\gk(z))^\al$. Then
\begin{equation}\label{hing}
H(h_{\xi}(z))\geq \tilde{C}\al(1-\al)(-\gk(z))^{\al-2}H(\gk(z)),
\end{equation}
where $\lambda(H(\gk(z)))\in \G$,
thus $h_{\xi}$ is $\G$-subharmonic.

Observe that 
$$h_{\xi}(z)\leq -\tilde{C}|z-\xi|^{2\al}\leq \varphi(z)-\varphi(\xi).$$
Thus $h_{\xi}(z)+\varphi(\xi)$ are local boundary barriers constructed following the method of Charabati from \cite{Cha} (in the paper \cite{BT1}, where the Monge-Amp\`ere case was considered, $h_{\xi}$ was simply chosen as $-(x_n)^\alpha$ in a suitable coordinate system, but this is not possible in the general case).

At this stage we recall that $u$ is bounded a priori by Theorem \ref{existence_general}. Hence we know that for some uniform constant $A$ one has $F[u]\leq A$ in the viscosity sense.

From the gathered information one can produce a global barrier for $u$ in a standard way (see \cite{BT1}). Indeed, consider the function $\tilde{h}(z):=sup_{\xi}\lbrace ah_{\xi}(z)+\varphi(\xi)\rbrace$ for a large but uniform constant $a$. Using the balayage procedure it is easy to show that $F(\tilde{h}(z))\geq A$ in the viscosity sense once $a$ is taken large enough. Thus $\tilde{h}$ majorizes $u$ by the comparison principle and so is a global barrier for $u$ matching the boundary data given by $\varphi$. By construction $\tilde{h}$ is globally $\al$-H\"older continuous.

Note on the other hand that $u$ is subharmonic as $\G\subset\G_1$, thus the harmonic extenstion $u_{\varphi}$ of $\varphi$ in $\Om$ majorizes $u$ from above. Recall that $u_{\varphi}$ is $\al$-H\"older continuous by classical elliptic regularity.

Coupling the information for both the lower and the upper barrier one obtains

\begin{equation}\label{bdry}
\forall z\in\overline\Om,\ \forall \xi\in\pa\Om\ \ |u(z)-u(\xi)|\leq K|z-\xi|^{\al}
\end{equation}

Denote by $K_1$ the quantity $K_2diam^2(\Om)max\lbrace1,  f({\bf 1})\rbrace+K$, where   ${\bf 1}=(1,\ldots,1)\in \R^n$ is the vector of the eigenvalues of the identity matrix, while $K_2:=\tilde{C} f({\bf 1})^{-1}$ and finally $\tilde{C}$ is the $\al$-Lipschitz constant of $\psi$. Consider for a small vector $\tau\in\C^n$ the function

$$v(z):=u(z+\tau)+K_2|\tau|^\al|z|^2-K_1|\tau|^\al$$
defined over $\Om_{\tau}:=\lbrace z\in\Om| z+\tau\in\Om\rbrace$.

It is easy to see by  using the barriers that if  $z+\tau\in\pa\Om$ or $z\in\pa\Om$ then

$$v(z)\leq u(z)+K|\tau|^{\al}+K_2diam^2{\Om}|\tau|^\al-K_1|\tau|^{\al}\leq u(z).$$

We now claim that $v(z)\leq u(z)$  in $\Om_{\tau}$. By the previous inequality this holds on $\pa(\Om_{\tau})$. Suppose the claim is false and consider the open subdomain $U$of $\Om_\tau$ defined by $U_{\tau}=\lbrace z\in\Om_{\tau}|\ v(z)>u(z)\rbrace$.

We will now prove that $v$ is a subsolution to $F[u]=f(\lambda(Hu))=\psi (z,u(z))$ in $U$. To this end pick a point $z_0$ and an upper differential test $q$ for $v$ at $z_0$. Observe then that $\tilde{q}(z):=q(z)-K_2|\tau|^\al|z|^2-K_1|\tau|^\al$ is then an upper differential test for $u(\tau+. )$ at the point $z_0$. Hence  
\begin{eqnarray*}
F[q (z_0)]&=&f(\la( H\tilde{q}(z_0))+K_2|\tau|^\al {\bf 1})\\
&\geq&f(\la(H\tilde{q}(z_0))+K_2|\tau|^\al f({\bf 1}) \\
&\geq& \psi(z_0+\tau, u(z_0+\tau))+K_2|\tau|^\al f({\bf 1}),
\end{eqnarray*}
where we have used the concavity and homogeneity of $f$ in the first inequality and the fact that $\tilde{q}$ is an upper differential test for $u(\tau+.)$ for the second one.

Next
\begin{eqnarray*}
&&\psi(z_0+\tau, u(z_0+\tau))+K_2|\tau|^\al  f({\bf 1})\\
&&\quad\geq  \psi(z_0+\tau, u(z_0+\tau)+K_2|\tau|^\al|z_0|^2-K_1|\tau|^\al) +K_2|\tau|^\al  f({\bf 1})\\
&&\quad =\psi (z_0+\tau,v(z_0))+K_2|\tau|^\al  f({\bf 1})\\
&&\quad \geq \psi (z_0+\tau, u(z_0))+K_2|\tau|^\al f({\bf 1}),
\end{eqnarray*}
where we have exploited twice the monotonicity of $\psi$ with respect to the second variable (and the fact that $z_0\in U_{\tau}$).
 
Exploiting now the H\"older continuity of $\psi$ with respect to the first variable we obtain
$$ \psi(z_0+\tau, u(z_0+\tau))+K_2|\tau|^\al  f({\bf 1})\geq \psi (z_0+\tau, u(z_0))+K_2|\tau|^\al  f({\bf 1})\geq \psi(z_0,u(z_0)).$$
This proves that $F[q(z_0)]\geq  \psi (z_0,u(z_0))$ and hence $F[v(z)]\geq  \psi(z,v(z))$ in the viscosity sense.

Thus over $U_{\tau}$, $v$ is subsolution and $u$ is a solution, which implies by comparison principle  that $v\leq u$ there, a contradiction unless the set  $U_{\tau}$ is empty.

We have thus proven that
$$\forall z\in\Om_{\tau}\ \ u(z+\tau)+K_2|\tau|^\al|z|^2-K_1|\tau|^\al\leq u(z),$$
which  implies the claimed $\alpha$- H\"older continuity.
\end{proof}

\section{Viscosity vs. pluripotential solutions}
  Let $\Omega$ be a bounded smooth strictly pseudoconvex domain in $\C^n$. Let $0<\psi\in C(\bar{\Omega}\times\R)$ be a continuous function
non-decreasing in the last variable.
 In this section, we study the 
relations between viscosity concepts with respect to the inverse $\sigma_k$ equations
\begin{equation}\label{vis.relation.eq}
\dfrac{(dd^cu)^n}{(dd^cu)^{n-k}\wedge\omega^k}= \psi (z,u) \quad\mbox{in}\quad\Omega,
\end{equation}
and pluripotential concepts  with respect to the equation
\begin{equation}\label{pp.relation.eq}
(dd^cu)^n= \psi(z,u)(dd^cu)^{n-k}\wedge\omega^k\quad\mbox{in}\quad\Omega.
\end{equation}
For the regular case, the following result was shown in \cite{GS}:
\begin{theorem}[Guan-Sun]\label{GS.the}
	Let $0<h\in C^{\infty}(\bar{\Omega})$ and $\varphi\in C^{\infty}(\partial\Omega)$. Then, there exists a smooth strictly
	plurisubharmonic function $u$ in $\bar{\Omega}$ such that
	\begin{equation}
	\dfrac{(dd^cu)^n}{(dd^cu)^{n-k}\wedge\omega^k}= h (z) \quad\mbox{in}\quad\Omega,\qquad
	u=\varphi\quad\mbox{in}\quad\partial\Omega.
	\end{equation}
\end{theorem}
Note that the function $u$ in Theorem \ref{GS.the} is a viscosity solution of \eqref{vis.relation.eq} 
 in the case when $\psi(z,u)=h(z)$. Using Theorem \ref{GS.the}, we obtain 
\begin{proposition}\label{vissol.the}
	If $u\in C(\bar \Omega)\cap PSH(\Omega)$ is a viscosity solution of \eqref{vis.relation.eq} then 
	there exists a sequence of smooth plurisubharmonic functions $u_j$ in $\Omega$ such that $u_j$ is decreasing to $u$
	and the function $\dfrac{(dd^c u_j)^n}{(dd^cu_j)^{n-k}\wedge\omega^k}$ converges uniformly to $\psi (z,u)$ as $j\rightarrow\infty$.
	In particular, $u$ is a solution of \eqref{pp.relation.eq} in the pluripotential sense.
\end{proposition}
\begin{proof}
	Let $\varphi_j\in C^{\infty}(\partial\Omega)$ and $0<\psi_j\in C^{\infty}(\bar{\Omega})$ be sequences of smooth functions such that
	$\varphi_j\searrow\varphi$ and $\psi_j\nearrow \psi(z,u)$ as $j\rightarrow\infty$. Then, by Theorem \ref{GS.the},
	for any $j=1,2,...$, there exists
	a smooth strictly plurisubharmonic function $u_j$ in $\bar{\Omega}$ such that
	\begin{equation}
	\dfrac{(dd^cu_j)^n}{(dd^cu_j)^{n-k}\wedge\omega^k}= \psi_j (z) \quad\mbox{in}\quad\Omega,\qquad
	u_j=\varphi_j\quad\mbox{in}\quad\partial\Omega.
	\end{equation}
	By the comparison principle, we have
	\begin{center}
		$u_1\geq u_2\geq ...\geq u_j\geq ...\geq u.$
	\end{center}
Let $C>\sup\limits_{\Omega}|z|^2$. By the homogeneity and the concavity of $S_{n,n-k}^{1/k}$, we have
\begin{center}
	$\dfrac{(dd^c(u_j+\epsilon |z|^2))^n}{(dd^c(u_j+\epsilon|z|^2))^{n-k}\wedge\omega^k}
	\geq \dfrac{(dd^cu_j)^n}{(dd^cu_j)^{n-k}\wedge\omega^k}
	+\epsilon^k.$
\end{center}
Then, by the comparison principle, for any $\epsilon>0$, there exists $N>0$ such that
\begin{center}
	$u_j+\epsilon (|z|^2-C)\leq u,$
\end{center}
for any $j>N$. Hence, $u_j$ is decreasing to $u$ as $j\rightarrow\infty$.
\end{proof}
Observe that a continuous solution of \eqref{pp.relation.eq} in the pluripotential sense may not be a viscosity solution of \eqref{vis.relation.eq}.
For example, if a continuous plurisubharmonic function $u: \Omega\rightarrow\R$ depends only on $n-k-1$ variables then $u$ is
a solution of  \eqref{pp.relation.eq} in the pluripotential sense but $u$ is not a  viscosity solution of \eqref{vis.relation.eq}.
Moreover, by Corollary \ref{relation.Cor1}, we know that a viscosity solution of \eqref{vis.relation.eq} has to sastisfy
$(dd^c u)^k\geq a\beta^k$ for some $a>0$.  The following question is natural:

\medskip
\noindent
{\bf Question.}\label{ppvis.ques}
	{\it If $u\in PSH(\Omega)\cap C(\bar{\Omega})$ satisfies \eqref{pp.relation.eq} in the pluripotential sense and
\begin{equation}
(dd^cu)^k\geq a\beta^k
\end{equation}	
for some $a>0$, does $u$ satisfy \eqref{vis.relation.eq} in the viscosity sense?}
	
\medskip	
	
At the end of this section, we will give the answer to a special case of this question.
 Now, we consider the relation between viscosity subsolutions of \eqref{vis.relation.eq}
  and \textit{pluripotential subsolutions} of \eqref{pp.relation.eq}. Recall that according to the definition in subsection 2.1 for any $n\times n$ complex matrix $A$ and $k\in\{1,...,n\}$,
   $S_k(A)$ denotes the coefficient with respect to $t^{n-k}$ of
the polynomial $\binom{n}{k}^{-1}\det (A+t {\rm Id}_n)$. 

Next we prove the following technical result:
\begin{lemma}\label{matrix.lem}
Assume that $A, B$ are $n\times n$ complex matrices and $k\in\{1,...,n\}$.  Then
	\begin{center}
		$S_k(AA^*)S_k(BB^*)\geq |S_k(AB^*)|^2.$
	\end{center}
\end{lemma}
\begin{proof}
Denote by $a_1,...,a_n$ and $b_1,...,b_n$, respectively, the row vectors of $A$ and $B$. Then
\begin{center}
	$S_k(AA^*)=(C_n^k)^{-1}\sum\limits_{\sharp J=k} \det\left(\langle a_p, a_q\rangle\right)_{p,q\in J},$
\end{center}
\begin{center}
	$S_k(BB^*)=(C_n^k)^{-1}\sum\limits_{\sharp J=k} \det\left(\langle b_p, b_q\rangle\right)_{p,q\in J},$
\end{center}
and
\begin{center}
	$S_k (AB^*)=(C_n^k)^{-1}\sum\limits_{\sharp J=k}\det\left(\langle a_p, b_q\rangle\right)_{p,q\in J}.$
\end{center}
We will show that, for any $J=\{p_1,...,p_k\}$ with $1\leq p_1<...<p_k\leq n$,
\begin{equation}\label{1.linear.eq}
\det\left(\langle a_p, a_q\rangle\right)_{p,q\in J}.
\det\left(\langle b_p, b_q\rangle\right)_{p,q\in J}\geq |\det\left(\langle a_p, b_q\rangle\right)_{p,q\in J}|^2.
\end{equation}
Indeed, if either $\{a_{p_1},...,a_{p_k}\}$ or
$\{b_{p_1},...,b_{p_k}\}$ are linearly dependent then both sides of \eqref{1.linear.eq} are equal to $0$.
Otherwise, exploiting the Gramm-Schmidt process, we can assume that $\{a_{p_1},...,a_{p_k}\}$ and 
$\{b_{p_1},...,b_{p_k}\}$ are  orthogonal systems (observe that the quantities in question do not change during the orthogonalization process). Next normalizing the vectors $a_{p_j}$ and $b_{p_j}$, $j=1,\cdots,n$ to unit length both sides change by the same factor. Hence it suffices to prove the statement for two collections of orthonormal bases.

 Under this assumption we have
\begin{eqnarray}
	\left(\langle a_p,a_q\rangle\right)_{p,q\in J}=\left(\langle b_p, b_q\rangle\right)_{p,q\in J}=Id_k.
\end{eqnarray}
Let $M=\left(\langle a_p, b_q\rangle\right)_{p,q\in J}$. Then $MM^*$ is semi-positive Hermitian matrix, and
\begin{eqnarray*}
{\rm Tr}(MM^*) 
	&=&\sum\limits_{l=1}^k\sum\limits_{j=1}^k|\langle b_{p_j}, a_{p_l} \rangle|^2\\
	&=&\sum\limits_{j=1}^k\langle b_{p_j}, \sum\limits_{l=1}^k\langle b_{p_j}, a_{p_l}\rangle a_{p_l}\rangle\\
	&\leq &\sum\limits_{j=1}^k \|b_{p_j}\|^2=k.
\end{eqnarray*}
Therefore, $|\det (M)|=\sqrt{\det (MM^*)}\leq 1$, hence we obtain \eqref{1.linear.eq}.
 Finally, using \eqref{1.linear.eq} and the Cauchy-Schwarz inequality, we infer that
$$
S_k(AA^*)S_k(BB^*)\geq |S_k(AB^*)|^2,
$$
as required.
\end{proof}

For any $n\times n$ Hermitian matrix $A=(a_{j\bar \ell})$, we denote
\begin{center}
	$\omega_A=\sum\limits_{j,\ell=1}^na_{j \bar \ell} \dfrac{i}{\pi}dz_j\wedge d\bar{z}_\ell,$
\end{center}
and
\begin{center}
	$\mathcal{B}(A, k):=\{B\in\mathcal{H}_+^n| \dfrac{\omega_B^k\wedge\omega_A^{n-k}}{\omega^n}=1\},$
\end{center}
where $k=1,2...,n$. 

\begin{theorem}\label{vissub.the}
	Let $u\in PSH(\Omega)\cap L_{loc}^{\infty}(\Omega)$ and $0<g\in C(\Omega)$. Then the following are equivalent:
	\begin{itemize}
		\item [(i)] $\dfrac{(dd^cu)^n}{(dd^cu)^{n-k}\wedge\omega^k}\geq g^k(z)$ in the viscosity sense.
		\item [(ii)] For all $B\in \mathcal{B} (Id, n-k)$, 
		\begin{center}
			$(dd^c u)^k\wedge\omega_{B^2}^{n-k}\geq g^k(z)\omega^n,$
		\end{center}
	in viscosity sense.
	\item [(iii)] For any open set $U\Subset\Omega$, there are smooth plurisubharmonic functions $u_{\epsilon}$
	and  functions $0<g^{\epsilon}\in C^{\infty}(U)$ such that $u_{\epsilon}$ are decreasing to  $u$ and $g^{\epsilon}$ converge uniformly
	to $g$ as $\epsilon\searrow 0$, and
\begin{equation}\label{iii.vissub.eq}
	(dd^c u_{\epsilon})\wedge\omega_{A_1}\wedge...\wedge\omega_{A_{k-1}}\wedge\omega_{B^2}^{n-k}\geq g^{\epsilon}\omega^n,
\end{equation}
pointwise in $U$ for any $B\in\mathcal{B}(Id, n-k)$ and $A_1,..., A_{k-1}\in\mathcal{B}(B^2, k)$.
\item  [(iv)]  For any open set $U\Subset\Omega$, there are smooth strictly plurisubharmonic functions $u_{\epsilon}$
and  functions $0<g^{\epsilon}\in C^{\infty}(U)$ such that the sequence $u_{\epsilon}$ is decreasing to  $u$ and the sequence $g^{\epsilon}$ converges uniformly
to $g$ as $\epsilon\searrow 0$, and
\begin{equation}\label{iv.vissub.eq}
\dfrac{(dd^cu_{\epsilon})^n}{(dd^cu_{\epsilon})^{n-k}\wedge\omega^k}\geq (g^{\epsilon})^k,
\end{equation}
pointwise in $U$ for any $B\in\mathcal{B}(Id, n-k)$.
	\end{itemize}
\end{theorem}
\begin{proof}
	$(iv\Rightarrow i)$ is obvious. It remains to show $(i\Rightarrow ii\Rightarrow iii\Rightarrow iv)$.
	
	\medskip
	\noindent
	$(i\Rightarrow ii)$ Assume that $q\in C^2$ is an upper test for $u$ from at $z_0\in\Omega$.
	 Then $q$ is strictly plurisubharmonic in a neighborhood of
	$z_0$ and
	\begin{center}
		$\dfrac{(dd^cq)^n}{(dd^cq)^{n-k}\wedge\omega^k}\geq g^k,$
	\end{center}
	at $z_0$.
	
	By using Lemma \ref{matrix.lem} for $\sqrt{Hq}$ and $(\sqrt{Hq})^{-1}B$, we have
	\begin{center}
		$\dfrac{(dd^cq)^{n-k}\wedge\omega^k}{(dd^cq)^n} \dfrac{(dd^cq)^k\wedge\omega_{B^2}^{n-k}}{\omega^n}
		=\dfrac{(dd^cq)^{n-k}\wedge\omega^k}{\omega^n} \dfrac{(dd^cq)^k\wedge\omega_{B^2}^{n-k}}{(dd^cq)^n}
		\geq \left(\dfrac{\omega_B^{n-k}\wedge\omega^k}{\omega^n}\right)^2,$
	\end{center}
	for any  $B\in \mathcal{H}_+^n$, (observe that $S_k(CC^*)= \dfrac{(dd^cq)^k\wedge\omega_{B^2}^{n-k}}{(dd^cq)^n}$
 and $S_k(\sqrt{Hq}C^*)=\dfrac{\omega_B^{n-k}\wedge\omega^k}{\omega^n}$  for $C=(\sqrt{Hq})^{-1}B$.)
	
	Then, for any $B\in\mathcal{B}(Id, n-k)$ we have
	\begin{center}
		$(dd^cq)^k\wedge\omega_{B^2}^{n-k}\geq g^k\omega^n,$
	\end{center}
at $z_0$. Hence
\begin{center}
	$(dd^c u)^k\wedge\omega_{B^2}^{n-k}\geq g^k\omega^n,$
\end{center}
in the viscosity sense.

\medskip
\noindent
$(ii\Rightarrow iii)$ Assume that $q\in C^2$ touches $u$ from above at $z_0\in\Omega$. Then, for any $B\in\mathcal{B}(Id, n-k)$,
\begin{center}
	$(dd^cq)^k\wedge\omega_{B^2}^{n-k}\geq g^k\omega^n,$
\end{center}
at $z_0$. 
 By the same arguments as in \cite{Lu}, we have
\begin{center}
	$(dd^cq)\wedge\omega_{A_1}\wedge...\wedge\omega_{A_{k-1}}\wedge\omega_{B^2}^{n-k}\geq g\omega^n,$
\end{center}
for any $B\in\mathcal{B}(Id, n-k)$, $A_1,..., A_{k-1}\in\mathcal{B}(B^2, k)$. Hence
\begin{equation}\label{1.subvispp.eq}
(dd^c u)\wedge\omega_{A_1}\wedge...\wedge\omega_{A_{k-1}}\wedge\omega_{B^2}^{n-k}\geq g\omega^n,
\end{equation}
in the viscosity sense  for any $B\in\mathcal{B}(Id, n-k)$, $A_1,..., A_{k-1}\in\mathcal{B}(B^2, k)$.

Let $g_j$ be a sequence of smooth functions in $\Omega$ such that $g_j\nearrow g$ as $j\rightarrow\infty$. Then
\begin{equation}\label{2.subvispp.eq}
(dd^c u)\wedge\omega_{A_1}\wedge...\wedge\omega_{A_{k-1}}\wedge\omega_{B^2}^{n-k}\geq g_j\omega^n,
\end{equation}
in the viscosity sense  for any $j\in\N$,  $B\in\mathcal{B}(Id, n-k)$ and  $A_1,..., A_{k-1}\in\mathcal{B}(B^2, k)$.
By the same arguments as in \cite{EGZ} (the proof of Proposition 1.5), $u$ satisfies \eqref{2.subvispp.eq}
in the sense of positive Radon measures. Using convolution to regularize $u$ and setting $u_{\epsilon}=u*\rho_{\epsilon}$,
we see that $u_{\epsilon}$ is smooth strictly plurisubharmonic and
\begin{center}
	$(dd^c u_{\epsilon})\wedge\omega_{A_1}\wedge...\wedge\omega_{A_{k-1}}\wedge\omega_{B^2}^{n-k}\geq (g_j)_{\epsilon}\omega^n,$
\end{center}
pointwise in $\Omega_{\epsilon}$. Choosing $g^{\epsilon}:= (g_{[1/ \epsilon]})_{\epsilon}$, we obtain \eqref{iii.vissub.eq}.

\medskip
\noindent
$(iii\Rightarrow iv)$
 At $z_0\in\Omega_{\epsilon}$, choosing
\begin{center}
	$B=\dfrac{Hu_{\epsilon}(z_0)}{(S_{n-k}(Hu_{\epsilon}(z_0)))^{1/(n-k)}}$
\end{center}
and
\begin{center}
	$A_1=A_2=...=A_{k-1}=\left(\dfrac{(dd^c u_{\epsilon}(z_0))^k\wedge\omega_{B^2}^{n-k}}{\omega^n}\right)^{-1/k} Hu_{\epsilon}(z_0) $,
\end{center}
we get,
\begin{eqnarray*}
	g^{\epsilon} &\leq& \left(\dfrac{(dd^c u_{\epsilon}(z_0))^k\wedge\omega_{B^2}^{n-k}}{\omega^n}\right)^{1/k}\\
	&=&\left(\dfrac{(dd^c u_{\epsilon}(z_0))^n}{\omega^n}\dfrac{1}{S_{n-k}(Hu_{\epsilon}(z_0))}\right)^{1/k}\\
	&=&\left(\dfrac{(dd^c u_{\epsilon}(z_0))^n}{\omega^n}\dfrac{\omega^n}{(dd^cu_{\epsilon})^{n-k}\wedge\omega^k}\right)^{1/k}\\
	&=&\left( \dfrac{(dd^cu_{\epsilon})^n}{(dd^cu_{\epsilon})^{n-k}\wedge\omega^k}\right)^{1/k},
\end{eqnarray*}
pointwise in $\Omega_{\epsilon}$. Then
\begin{center}
	$\dfrac{(dd^cu_{\epsilon})^n}{(dd^cu_{\epsilon})^{n-k}\wedge\omega^k}\geq (g^{\epsilon})^k.$
\end{center}
The proof is completed.
\end{proof}

As a consequence, our result  implies that a viscosity subsolution is a pluripotential subsolution.
\begin{theorem}\label{relation.Cor1}
	Assume that $\psi(z,s)=\psi(z)$ with $\psi\in C^0(\Omega)$ and $u\in PSH(\Omega)\cap L_{loc}^{\infty}(\Omega)$
	 is a viscosity subsolution of \eqref{vis.relation.eq}. Then
	 \begin{equation}\label{1.corvispp.eq}
	 (dd^cu)^n\geq \psi (dd^cu)^{n-k}\wedge\omega^k,
	 \end{equation}
and
	\begin{equation}\label{2.corvispp.eq}
	(dd^c u)^k\geq \binom{n}{k}^{-1}\psi \omega^k,
	\end{equation}	
in the pluripotential sense.
	If $u$ is continuous then the conclusion still holds in the case where $\psi$ depends on both variables.
\end{theorem}
\begin{proof}
By Theorem \ref{vissub.the},  for any open set $U\Subset\Omega$,  there are strictly plurisubharmonic functions $u_{\epsilon}\in C^{\infty}(U)$
	and functions $0<h^{\epsilon}\in C^{\infty}(U)$ such that $u_{\epsilon}$ is decreasing to  $u$ and $h^{\epsilon}$ converges uniformly
	to $\psi$ as $\epsilon\searrow 0$, and
	\begin{equation}
	\dfrac{(dd^cu_{\epsilon})^n}{(dd^cu_{\epsilon})^{n-k}\wedge\omega^k}\geq h^{\epsilon},
	\end{equation}
	pointwise in $U$. 
	Choosing $B={\rm Id}_n$ and letting $\epsilon\rightarrow 0$, we obtain \eqref{1.corvispp.eq}.
	
\medskip	
	It also follows from Theorem \ref{vissub.the} that we can choose $u_{\epsilon}$ and $h^{\epsilon}$ so that
	\begin{equation}\label{1.corvisppproof.eq}
	(dd^c u_{\epsilon})^{k}\wedge\omega_{B^2}^{n-k}\geq h^{\epsilon}\omega^n,
	\end{equation}
	pointwise in $U$ for any $B\in\mathcal{B}({\rm Id}, n-k)$.
	Fix $z_0\in U$ and $0<\epsilon\ll 1$.
	 We can choose complex coordinates so that $Hu_{\epsilon}(z_0)={\rm diag} (\lambda_1,\ldots,\lambda_n)$, where
	 $0\leq\lambda_1\leq \ldots\leq\lambda_n$. Choosing
	 \begin{center}
	 	$B=\binom{n}{k}^{1/(n-k)}{\rm diag} (0,\ldots,\underset{k-th}{\underbrace{0}}, 1,\ldots,1),$
	 \end{center}
	 we get
	 \begin{center}
	 	$\lambda_1\ldots \lambda_k\geq \binom{n}{k}^{-1} h^{\epsilon}.$
	 \end{center}
	 Then
	  \begin{center}
	  	$(dd^c u_{\epsilon})^{k}\geq\binom{n}{k}^{-1} h^{\epsilon}\omega^k$,
	  \end{center}
	  	pointwise in $U$. Letting $\epsilon\rightarrow 0$, we obtain \eqref{2.corvispp.eq}.
\end{proof}
{\bf Remark} Note that for strictly positive $\psi$ (\ref{2.corvispp.eq}) implies that the {\it natural} space of functions to consider for the Hessian quotient problem  \eqref{vis.relation.eq} is {\it not} the space of bounded plurisubharmonic functions but a considerably smaller one.

By assuming an additional conditions, we can also prove that a pluripotential subsolution is a visocsity one.
\begin{corollary}\label{relation.Cor2}
	Assume that $\psi(z,s)=\psi(z)>0$ with $\psi\in C^0(\Omega)$ and $u$ is a local bounded plurisubharmonic function in $\Omega$ satisfying
	\begin{center}
		$(dd^cu)^k\geq \psi \omega^k,$
	\end{center}
	in the pluripotential sense. Then
\begin{center}
	$\dfrac{(dd^c u)^n}{(dd^cu)^{n-k}\wedge\omega^k}\geq \psi ,$
\end{center}
in the viscosity sense.
\end{corollary}
\begin{proof}
	By the assumption, for any $A\in\mathcal{H}_+^n$,
	\begin{equation}\label{1cor2proof.eq}
		(dd^cu)^k\wedge\omega_A^{n-k}\geq \psi
		   \omega^k\wedge\omega_A^{n-k},
	\end{equation}
		in the pluripotential sense. By \cite{Lu}, \eqref{1cor2proof.eq} also holds in the viscosity sense. If $A=B^2$ for some
		$B\in\mathcal{B}({\rm Id}, n-k)$ then, by using Lemma \ref{matrix.lem}, we have
		\begin{center}
			$\omega^k\wedge\omega_{B^2}^{n-k}
\geq \left(\dfrac{\omega_B^{n-k}\wedge\omega^k}{\omega^n} \right)^2\omega^n=\omega^n.$
		\end{center}
		Then
		\begin{center}
			$(dd^cu)^k\wedge\omega_{B^2}^{n-k}\geq \psi\omega^n,$ 
		\end{center}
		in the viscosity sense, for any  $B\in\mathcal{B}({\rm Id}, n-k)$. Applying Theorem \ref{vissub.the}, we obtain
		\begin{center}
			$\dfrac{(dd^c u)^n}{(dd^cu)^{n-k}\wedge\omega^k}\geq \psi,$
		\end{center}
		in the viscosity sense.
\end{proof}

We now discuss  the notion of a {\it supersolution}.  By the same argument as in  \cite{GLZ1}, (relying on  the Berman's idea from \cite{Ber}) we obtain the following relation between  viscosity supersolutions of \eqref{vis.relation.eq}
and pluripotential supersolutions  of \eqref{pp.relation.eq}:
\begin{proposition}\label{vissup.the}
	Let $u\in PSH(\Omega)\cap C(\bar{\Omega})$ be a viscosity supersolution of \eqref{vis.relation.eq}.
	Then there exists an increasing sequence of strictly psh functions $u_j\in C^{\infty}(\bar{\Omega})$ such that $u_j$ converges 
	in capacity	to $u$ as $j\rightarrow\infty$, and
	\begin{center}
		$\dfrac{(dd^cu_j)^n}{(dd^cu_j)^{n-k}\wedge\omega^k}\leq \psi (z,u),$
	\end{center}
	pointwise in $\Omega$. In particular, 
	\begin{center}
		$(dd^cu)^n\leq \psi (z,u)(dd^cu)^{n-k}\wedge\omega^k,$
	\end{center}
	in the pluripotential sense.\\
	If there exists $a>0$ such that $(dd^c u)^k\geq a\omega^k$ then $u_j$ can be chosen such that 
		\begin{center}
			$\dfrac{(dd^cu_j)^n}{(dd^cu_j)^{n-k}\wedge\omega^k}\geq b,$
		\end{center}
		pointwise in $\Omega$ for some $b>0$.
\end{proposition}
For the definition  of convergence in capacity, we refer to \cite{GZ2} and references therein.  
\begin{proof}
Denote $\varphi= u|_{\partial\Omega}$ and $g(z)=\psi(z,u(z))$. Then, for any $j\geq 1$, there exists a unique viscosity solution
$v_j$ of
\begin{equation}\label{1.vissupproof.eq}
\begin{cases}
\dfrac{(dd^cv_j)^n}{(dd^cv_j)^{n-k}\wedge\omega^k}= e^{j(v_j-u)}g(z) \quad\mbox{in}\quad\Omega,\\
v_j=\varphi \quad\mbox{in}\quad\partial\Omega.
\end{cases}
\end{equation}
Applying the comparison principle to the equation
\begin{center}
	$\dfrac{(dd^cv)^n}{(dd^cv)^{n-k}\wedge\omega^k}= e^{j(v-u)}g(z),$
\end{center}
we get $u\geq v_j$ and $v_{j+1}\geq v_j$ for any $j\geq 1$.\\
Note that, by Theorem \ref{vissol.the},
\begin{center}
	$(dd^cv_j)^n=e^{j(v_j-u)}g(z)(dd^cv_j)^{n-k}\wedge\omega^k,$
\end{center}
in the pluripotential sense. For any $h\in PSH(\Omega)$ such that $-1\leq h\leq 0$, we have,
\begin{eqnarray*}
\epsilon^n\int\limits_{\{v_j<u-2\epsilon\}}(dd^c h)^n &\leq &\int\limits_{\{v_j<u+\epsilon h-\epsilon\}}(dd^c (u+\epsilon h))^n\\
	&\leq& \int\limits_{\{v_j<u+\epsilon h-\epsilon\}}(dd^c v_j)^n\\
	&\leq &\int\limits_{\{v_j<u-\epsilon\}} e^{j(v_j-u)}g(z)(dd^cv_j)^{n-k}\wedge\omega^k\\
	&\leq &e^{-j\epsilon}\int\limits_{\{v_1<u-\epsilon\}} g(z)(dd^cv_j)^{n-k}\wedge\omega^k\\
	&\leq& Ce^{-j\epsilon},
\end{eqnarray*}
where $C>0$ is independent of $j$. The last inequality holds by the  Chern-Levine-Nirenberg inequalities (cf. \cite{GZ2}).
This implies that $v_j$ converges to $u$ in capacity.

\medskip
 If there exists $a>0$ such that $(dd^c u)^k\geq a\omega^k$ then, by Corollary \ref{relation.Cor2}, 
$$
 	\dfrac{(dd^cu)^n}{(dd^cu)^{n-k}\wedge\omega^k}\geq a,
$$
 in the viscosity sense. Choosing $M\gg 1$ such that $e^{-M}\sup\limits_{\Omega} g<a$, we get
$$
 	\dfrac{(dd^c v_j)^n}{(dd^cv_j)^{n-k}\wedge\omega^k}\leq a e^{j(v_j-u)+M}.$$
 Applying the comparison principle to the equation
 $$\dfrac{(dd^cv)^n}{(dd^cv)^{n-k}\wedge\omega^k}= ae^{j(v-u)},$$
 we get $v_j+\dfrac{M}{j}\geq u$ for any $j\geq 1$. Then
$$\dfrac{(dd^cv_j)^n}{(dd^cv_j)^{n-k}\wedge\omega^k}=e^{j(v_j-u)}g(z)\geq e^{-M}g(z),$$
 for any $j\geq 1$. Hence, by Corollary \ref{relation.Cor1}, 
 \begin{center}
 $(dd^c v_j)^k\geq \binom{n}{k}^{-1} e^{-M}g(z)\geq \binom{n}{k}^{-1}e^{-M}\min\limits_{\bar{\Omega}}g ,$	
 \end{center}
  for any $j\geq 1$.\\
Now, by Theorem \ref{vissol.the}, for any $j$ we can choose a strictly plurisubharmonic function $u_j\in  C^{\infty}(\bar{\Omega})$,
such that $$ v_j-\dfrac{1}{2^j}\leq u_j \leq v_j-\dfrac{1}{2^{j+1}}$$ and 
\begin{center}
$-\dfrac{1}{2^j}\leq\dfrac{(dd^cu_j)^n}{(dd^cu_j)^{n-k}\wedge\omega^k} - e^{j(v_j-u)}g(z)\leq 0.$
\end{center}
 It is easy to see that $u_j$ satisfies the required properties.
\end{proof}

The next result gives the answer to a special case of Question \ref{ppvis.ques}:
\begin{theorem}\label{ppvissol.the}
Let $u\in PSH(\Omega)\cap C(\Omega)$ such that
\begin{equation}\label{vissup.ppvis.the.eq}
	\dfrac{(dd^cu)^n}{(dd^cu)^{n-k}\wedge\omega^k}\leq \psi(z,u),
\end{equation}
in the viscosity sense and
\begin{equation}\label{ppsub.ppvis.the.eq}
	(dd^cu)^n\geq \psi (z,u)(dd^cu)^{n-k}\wedge\omega^k,
\end{equation}
in the pluripotential sense. If there exists $a>0$ such that $(dd^c u)^k\geq a\omega^k$ then $u$ is a viscosity solution of the equation
\begin{equation}\label{conclusion.ppvis.the.eq}
\dfrac{(dd^cu)^n}{(dd^cu)^{n-k}\wedge\omega^k}= \psi (z,u).
\end{equation}
\end{theorem}
\begin{proof}
	It remains to show that $u$ is a viscosity subsolution of \eqref{conclusion.ppvis.the.eq} in any smooth strictly pseudoconvex domain
	$U\Subset\Omega$.
	
\medskip	
	 Let $V$ be a smooth strictly pseudoconvex domain such that $U\Subset V\Subset\Omega$. By Proposition \ref{vissup.the},
	  there exists an increasing sequence of strictly plurisubharmonic functions $u_j\in  C^{\infty}(\bar{V})$, such that $u_j$ converges 
	 in capacity to $u$ as $j\rightarrow\infty$, and
	 $$
	 b\leq\dfrac{(dd^cu_j)^n}{(dd^cu_j)^{n-k}\wedge\omega^k}\leq \psi (z,u),$$
	 pointwise in $V$, where $b>0$. By Corollary \ref{relation.Cor1}, we have $(dd^c u_j)^k\geq \binom{n}{k}^{-1}b \omega^k$.
	  Then, there exists $C>0$ such that
	 \begin{center}
	 	$(dd^cu_j)^{n-k}\wedge\omega^k\geq\dfrac{1}{\psi (z,u)}(dd^cu_j)^n\geq C\omega^n.$
	 \end{center}
	  Denote
	 \begin{center}
	 	$f_j(z):=\dfrac{(dd^cu_j)^n}{(dd^cu_j)^{n-k}\wedge\omega^k}.$
	 \end{center}
 Then $f_j(z)\leq \psi(z,u)$ for any $z\in V$, and $(\psi-f_j) (dd^cu_j)^{n-k}\wedge\omega^k\geq C(\psi -f_j)\omega
^n$ converges weakly to $0$.
 Hence $f_j$ converges in Lebesgue measure to $\psi$ in $V$ as $j\rightarrow\infty$.\\
 Now, by Theorem \ref{vissub.the}, we have
 \begin{center}
 $(dd^c u_j)\wedge\omega_{A_1}\wedge...\wedge\omega_{A_{k-1}}\wedge\omega_{B^2}^{n-k}\geq (f_j)^{1/k}\omega^n,$
 \end{center}
 pointwise in $V$ for any $B\in\mathcal{B}(Id, n-k)$ and $A_1,..., A_{k-1}\in\mathcal{B}(B^2, k)$. 
 Letting $j\rightarrow\infty$, we get
 \begin{center}
 	$(dd^c u)\wedge\omega_{A_1}\wedge...\wedge\omega_{A_{k-1}}\wedge\omega_{B^2}^{n-k}\geq \psi^{1/k}\omega^n,$
 \end{center}
 in the sense of Radon measures. It follows from \cite{Lu} that
$$(dd^c u)^k\wedge\omega_{B^2}^{n-k}\geq \psi^{1/k} \omega^n,
$$
 in the viscosity sense. Using Theorem \ref{vissub.the}, we get that $u$ is a viscosity subsolution of \eqref{conclusion.ppvis.the.eq}
 in $U$. The proof is completed.
\end{proof}

\section{Dirichlet problem for the Lagrangian phase operator}
\label{lagrangian_section}
 In this section, we prove the existence of unique viscosity solution to the Dirichlet problem for the Lagrangian phase operator. The existence and uniqueness of the smooth version was obtained recently by Collins-Picard-Wu \cite{CPW}. Let $\Omega\subset\C^n$ be a bounded domain. Consider the Dirichlet problem
\[
   \left\{ \begin{aligned}
   F[u]: = \sum_{i=1}^n\arctan \lambda_i&=h(z),\, \text{ on } \Omega \\
     u&=\varphi  \, \text{ on }\partial \Omega.
   \end{aligned} 
   \right.  \tag*{$($LA$)$} \label{lagrangian}
\] 
where $\lambda_1,\ldots,\lambda_n$ is the eigenvalues of the complex Hessian $Hu$. We can also write $F[u]=f(\lambda(Hu))$.  We assume that $\varphi \in C^0(\partial \Omega)$ and $h:\bar \Omega\rightarrow [(n-2){\pi\over 2} +\delta, n{\pi\over 2})$ is continuous, for some $\delta>0$.

\medskip  The Lagrangian phase operator  $F$ in \ref{lagrangian} arises in geometry and mathematical physics. We refer to \cite{CPW, HL2,JY,CJY, Y,WY1,WY2} and references therein for the details.  

\medskip
Since $h\geq (n-2)\frac{\pi}{2}$, this case is called the {\it supercritical phase} following \cite{Y,JY,CJY,CPW}. Recall first the following properties (cf. \cite{Y,WY2, CPW});
\begin{lemma}
\label{properties}
Suppose $\lambda_1\geq \lambda_2\geq \ldots\geq \lambda_n$ satisfying $\sum_i \arctan \lambda_i\geq (n-2){\pi\over 2}+\delta$ for some $\delta>0$. Then we have 
\begin{itemize}
\item [(1)]  $\lambda_1\geq\lambda_2\geq \ldots\geq \lambda_{n-1}>0$ and $|\lambda_{n}|\leq \lambda_{n-1}$,

\item [(2)] $\sum_i \lambda_i\geq 0$, and $\lambda_n\geq -C(\delta)$,

 \item [(3)]  $\sum \lambda_i^{-1}\leq -\tan(\delta)$ when $\lambda_n<0$.
 \item [(4)] for any $\sigma\in ((n-2){\pi\over 2} , n{\pi\over 2})$,  the set
 $\Gamma^\sigma:=\{\lambda\in \R^n\, | \, \sum_i \arctan\lambda_i>\sigma \}$
 is a convex set and $\partial \Gamma^\sigma$ is a smooth convex hypersurface.
\end{itemize}
\end{lemma}

It follows form Lemma \ref{properties} that the function $f$ can be defined on a cone $\Gamma$ satisfying $ \Gamma_n\subset \Gamma \subset \Gamma_1 $.  We also remark that if $h\geq (n-1)\frac{\pi}{2}$, then $F$  is concave while $F$  have concave level sets if $(n-2)\frac{\pi}{2} h\leq (n-1)\frac{\pi}{2}$, but in general $F$ may not be   concave  (cf. \cite{CPW}). Therefore we can not apply Theorem \ref{comprison_principle} directly. Fortunately, we still have a comparison principle for the Lagrangian operator using Lemma \ref{general_comparision}.
\begin{lemma} 
\label{lagrange_comparison}
Let $u,v\in L^\infty(\overline{\Omega})$ be viscosity subsolution and supersolution of equation $F[u]=f(\lambda(Hu))=h$ on $\Omega$.  Then
\begin{equation}\label{comparision_2}
\sup_{\Omega}(u-v)\leq \max_{\partial \Omega} \left\lbrace  (u-v)^*,0\right\rbrace.
\end{equation}
\end{lemma}
\begin{proof}
We first define $\epsilon>0$ by $\max_{\bar{\Omega}}h=n\frac{\pi}{2}-\epsilon$.  Now for any $0<\tau\leq \epsilon/2$, set
$u_\tau=u+\tau |z|^2.$   Let $q_\tau$ be any upper test for $u_\tau$ at any point $z_0\in\Omega$, then $q=q_\tau-\tau |z|^2 $ is also an upper test for $u$ at $z_0$. By the definition we have
$$F[q](z_0)=\sum_{i=1}^n \arctan \lambda_i (z_0)\geq h(z_0),$$
where $\lambda(z_0)=\lambda (Hq(z_0))$.
We also have
\begin{equation}
F[q_\tau](z_0)=\sum_{i=1}^n \arctan (\lambda_i(z_0) +\tau ).
\end{equation}
Next,  if $F[q](z_0) \geq n {\pi\over 2}-\frac{\epsilon}{2}$, then $ F[q](z_0) \geq h(z_0)+\frac{\epsilon}{2}$ hence 
\begin{equation}
\label{bound_upper_test}
F[q_\tau](z_0) \geq h(z_0)+\frac{\epsilon}{2}.
\end{equation}
Conversely, if $F[q](z_0)<n{\pi\over 2}-\frac{\epsilon}{2} $, this implies that $\arctan(\lambda_n(z_0))\leq  \frac{\pi}{2}-\frac{\epsilon}{2n}$.  Combining with Lemma  \ref{properties} (2), we get $-C(\delta)\leq \lambda_n(z_0)\leq C(\epsilon)$. Using the Mean value theorem, there exists $\hat \lambda_n \in (\lambda_n(z_0),\lambda_n(z_0)+\tau)$ such that
$$\arctan(\lambda_n(z_0)+\tau)- \arctan \lambda_n(z_0) = \frac{1}{1+ \hat \lambda_n^2} \tau\geq  C(\delta,\epsilon,\tau)>0.$$ 
It follows that 
\begin{equation}
F[q_\tau](z_0)\geq  F[q](z_0)+C(\delta,\epsilon,\tau) \geq h(z_0)+ C(\delta,\epsilon,\tau).
\end{equation}
Combing with (\ref{bound_upper_test}) yields
$$F[q_\tau](z_0)\geq    h(z_0)+ C,$$
where $C>0$ depending only on $\delta,\epsilon,\tau$. We thus infer that $u_\tau$ satisfies $ F[u_\tau]\geq h(z)+C$ in the viscosity sense. Therefore applying Lemma \ref{general_comparision} to $u_\tau$ and $ v$, then let $\tau \rightarrow 0$, we obtain the desired inequality.  
\end{proof}

\begin{theorem}
Let $\Omega$ is a bounded  $C^2$ domain. Let  $\underline u$ is an bounded upper semi-continuous function on $ \Omega $ satisfying $F[\underline u]\geq h(z)$ in $\Omega$ in the viscosity sense and $\underline{u} =\varphi$ on $\partial \Omega$.   Then the Dirichlet problem \ref{lagrangian} admits a unique viscosity solution $u\in C^0(\Omega)$.

\begin{proof}
 It suffices to find a viscosity supersolution $\bar u$ for the equation $F[u]=h(z)$, satisfying $\bar u=\varphi$ on $\partial \Omega$.  The $C^2$-boundary implies the existence of a harmonic function  $\phi$ on $\Omega$ for arbitrary given continuous boundary data $\f$. Since  $\sum_i \lambda_i(H\phi)=0$,  it follows from Lemma \ref{properties} that we have $F[\phi]<(n-2)\frac{\pi}{2} +\delta\leq h $, hence $\phi$ is a supersolution for \ref{lagrangian}.  The rest of the proof is similar to the one of Theorem \ref{existence_general}, by using Lemma \ref{lagrange_comparison}.
\end{proof}

\end{theorem}

Address:

S\l awomir Dinew: Institute of Mathematics, Jagiellonian University, ul \L ojasiewicza 6, 30-348 Krak\'ow, Poland\\
{\tt e-mail: slawomir.dinew@im.uj.edu.pl}\\

Hoang-Son Do: Institute of Mathematics, Vietnam Academy of Science and Technology, 18 Hoang Quoc Viet, Cau Giay, 100000 Hanoi, Vietnam\\
{\tt email:  hoangson.do.vn@gmail.com} \\

Tat-Dat T\^o: Institut Math\'ematiques de Toulouse, Universit\'e Paul Sabatier, 31062 Toulouse cedex 09, France.\\
{\tt e-mail: tat-dat.to@math.univ-toulouse.fr}
\end{document}